\documentclass[letter,12pt]{amsart}
\usepackage{extsizes}
\usepackage{blindtext}
\usepackage[utf8]{inputenc}
\usepackage{mathrsfs} 
\usepackage{amsmath}
\usepackage{amssymb}
\usepackage{amsthm}
\usepackage{pst-node}
\usepackage{blindtext}
\usepackage[pdftex,%
    nomarginpar,%
    lmargin=1in,%
    rmargin=1in,%
    tmargin=1in,%
    bmargin=1in,%
]{geometry}
 \usepackage{float}
 \usepackage{amssymb}
\usepackage{tikz-cd}
\usepackage{amscd}
\usepackage[all,color]{xy}
\usepackage{sseq}
\usepackage{amsfonts}
\usepackage{enumerate}
\usepackage{graphicx}
\usepackage{fancyhdr}
\usepackage{tikz}
\usepackage{graphics}
\usepackage{hyperref}
\usepackage[title]{appendix}
\usepackage[numbers]{natbib}
\usepackage{quiver}

\theoremstyle{plain} \numberwithin{equation}{section}

\newtheorem{theorem}{Theorem}[section]

\newtheorem{corollary}[theorem]{Corollary}

\newtheorem{lemma}[theorem]{Lemma}

\newtheorem{proposition}[theorem]{Proposition}
\theoremstyle{definition}
\newtheorem{definition}[theorem]{Definition}

\newtheorem{remark}[theorem]{Remark}
\newtheorem{example}[theorem]{Example}

\newcommand{\C}{\mathcal{C}}
\newcommand{\D}{\mathcal{D}}
\newcommand{\G}{\mathcal{G}}
\newcommand{\M}{\mathcal{M}}

\def \f-{f^{-1}}
\def \fp-{f^{-1}_\partial}

\usepackage{hyperref}
\usepackage{caption}
\usepackage{subcaption}

\title{Concentration structures on categories and horizontal categorification}
\author{Yangxiao Luo, Shunyu Wan}
\date{}
\usepackage{natbib}
\usepackage{graphicx}

\begin{document}

\maketitle

\begin{abstract}

We introduce a theory for encoding and manipulating algebraic data on categories via \textit{concentration structures}, which are equivalence relations on morphisms that satisfy certain axioms. For any category with a concentration structure we can functorially construct a \textit{concentration monoid}, which can be used to give a precise definition of horizontal categorification and decategorification. Moreover, by studying concentration structures on fundamental groupoids, we show that every group arises as the concentration monoid of a trivial category, up to category equivalence.


\end{abstract}

\section{Introduction}

The study of additional structures on categories has long provided new perspectives on both category theory and its applications to geometry and topology. Classical enhancements, such as monoidal structures, enrichment \cite{MR651714}, or higher-categorical refinements \cite{leinster1998basicbicategories} \cite{MR2522659}, allow one to extract algebraic invariants from categories or to model geometric phenomena within a categorical framework. One of the most significant developments arsing from this viewpoint is categorification.



The usual (vertical) categorification was introduced by Crane \cite{Crane:1995qj, MR1295461}, as the process of replacing set-theoretic structures by higher categorical analogues so that equalities are replaced by isomorphisms and functions by functors. In contrast, the process of decategorification quotients out higher categorical structures while retaining the isomorphism classes of lower categorical structures.

Vertical categorification has appeared across representation theory, topology, and geometry. One of the most notable examples is the categorification of Jones polynomial, leading to Khovanov homology \cite{MR1740682}. Vertical decategorification has also been widely studied, for example, trace decategorification \cite{MR3319701} \cite{MR3448186}, the decategorification of bordered Khovanov homology \cite{MR3312621}, the decategorification of bordered knot Floer homology \cite{MR3900777} and the decategorification of sutured Floer homology \cite{MR2805998}.

On the other hand, horizontal categorification, or oidification, is typically understood as the process of replacing one-object categories by multi-object categories, so that the endomorphsisms on the single object are replaced by the morphisms between different objects \cite{MR2494910} \cite{MR2735769}. For example, categories (resp. groupoids) can be regarded as the horizontal categorification of monoids (resp. groups).

Contrary to vertical categorification, horizontal categorification and decategorification have received relatively little attention and, to the best of the authors' knowledge, still lack a precise definition. In this paper, we introduce \emph{concentration structures} on categories, which are equivalence relations on morphisms satisfying certain axioms (Definition \ref{def: concentration structure}). Such a structure on a category has a naturally associated monoid, which can be defined as a horizontal decategorification of the original category.

In the following subsections, we will give a brief overview of concentration structures and some of the applications.

\subsection{Concentration structures and concentration preserving functors}

Let $\mathcal{C}$ be a small category. The sets of objects and morphisms of $\mathcal{C}$ are denoted as $\textup{Ob}_{\mathcal{C}}$ and $\textup{Mor}_{\mathcal{C}}$.

\begin{definition} \label{def: concentration structure}
     A \textbf{concentration structure} on $\mathcal{C}$ is an equivalence relation $\sim$ on $\textup{Mor}_{\mathcal{C}}$, which satisfies the following axioms:
    
    \begin{enumerate}
        \item $id_A \sim id_B$ for any objects $A, B \in \textup{Ob}_{\mathcal{C}}$. (Identity axiom)
        
        \item If $f \sim f', g \sim g'$ and both $f\circ g, f' \circ g'$ exist, then $f \circ g \sim f' \circ g'$. (Composition axiom)

\item For any morphisms $f, g \in \textup{Mor}_{\mathcal{C}}$, there exist some $f' \sim f$ and $ g' \sim g$ such that $f' \circ g'$ exists. (2-Existence axiom)
        \item  For any morphisms $f\sim f', g\sim g', h\sim h', m\sim f\circ g, n \sim g'\circ h'$ we have $f' \circ n \sim m \circ h$ whenever all the compositions are  well defined. (Associativity axiom)
    \end{enumerate}
   
   If $\sim$ is a concentration structure on $\C$, we will say $(\C,\sim)$ is a \textbf{category with concentration}.
\end{definition}



Any category has a \textbf{trivial concentration structure}, such that $f \sim_{tr} g$ for any morphisms $f$ and $g$. If a category has only one object, then it admits a \textbf{discrete concentration structure}, i.e. $f \sim_{dis} g$ if and only if $f = g$. For nontrivial and non-discrete examples of concentration structure, see Example \ref{ex: Z/2 and Z/4 category with concentration} and Section \ref{subsec: mroe examples}.

To construct concentration structures from existing ones, we consider the following question: given a functor $F:\C \to \D$ and a concentration structure $\sim_\D$ on $\D$, can $\sim_\D$ be pulled back to $\C$ along $F$? The answer is in general negative, but it is affirmative when $F$ is a so-called $2$-lifting functor (Definition \ref{def: $2$-lifting functor}), which include, for example, surjective Grothendieck fibrations. Moreover, Theorem \ref{thm: Every concentration structure is pullback} shows that every concentration structure arises as a pullback of a discrete concentration structure. 

\vspace{5mm}

To work with categories equipped with concentration, we require functors that preserve the concentration structures. This leads to the natural definition of concentration preserving functors.

\begin{definition} \label{def: concentration preserving functor}
    We say a map $F: (\C,\sim_{\C}) \rightarrow(\D,\sim_{\D}) $ is a \textbf{concentration preserving functor} if the following holds 
    
    \begin{itemize}
        \item $F$ is a functor between $\C$ and $\D$ in the usual sense.
        \item For any $f,f'\in \textup{Mor}_{\C}$, if $f\sim_{\C}f'$ then $F(f) \sim_{\D}F(f')$. 
    \end{itemize}

    $F$ is called a \textbf{concentration isomorphism} if $F$ is a strongly invertible functor ($F$ has a strong inverse $F^{-1}$ such that $F^{-1} \circ F = id_{\mathcal{C}}$ and $F \circ F^{-1} = id_{\mathcal{D}}$), and $F, F^{-1}$ are both concentration preserving. We say $(\C,\sim_{\C})$ and $(\D,\sim_{\D})$ are isomorphic, denoted as $(\C,\sim_{\C}) \cong (\D,\sim_{\D})$, if there exists a concentration isomorphism between them.
\end{definition}

\subsection{Concentration monoids}

Our definition of concentration structures is similar to the generalized congruence relation on categories \cite{MR1725510GeneralizedCongruences}, but we consider morphisms instead of sequences of morphisms, and we impose additional axioms. These extra axioms guarantee that any category with concentration determines a concentration monoid.

\begin{definition}
    Given $(\mathcal{C},\sim)$, a category with concentration, define the \textbf{concentration monoid} of $(\mathcal{C}, \sim)$ to be $M_{(\mathcal{C}, \sim)} = \textup{Mor}_{\mathcal{C}} / \sim$, with multiplication 
    $$
    [f][g] = [f' \circ g']
    $$
    where $f' \sim f, g' \sim g$ and $f' \circ g'$ exists. The third axiom of concentration structure guarantees that we can always find such $f'$ and $g'$.
\end{definition}

Let $\mathscr{C}at_{\sim}$ (resp. $\mathscr{G}rpd_{\sim}$) be the category of small categories (resp. the category of groupoids) with concentration structures whose morphisms are concentration preserving functors. Denote the category of monoids (resp. the category of groups) to be $\mathscr{M}on$ (resp. $\mathscr{G}rp$). The next theorem tells us that taking concentration monoid is functorially well-behaved.

\begin{theorem} \label{thm: Taking concentration monoid is a functor}
    Taking concentration monoid is a functor $\mathbf{M}: \mathscr{C}at_{\sim} \to \mathscr{M}on$. When restricted on $\mathscr{G}rpd_{\sim}$, it is a functor $\mathbf{M}: \mathscr{G}rpd_{\sim} \to \mathscr{G}rp$.
\end{theorem}

We will give the precise definition of $\mathbf{M}$, verify its well-definedness in Proposition \ref{prop: G is a monoid} and Proposition \ref{prop: concentration preserving functor induces monoid homomorphism}, before proving Theorem \ref{thm: Taking concentration monoid is a functor} in Section \ref{sec: well-definedness and functoriality of taking concentration monoid}. 

\vspace{5mm}

By considering the discrete concentration structure on the canonical one-object category associated with a concentration monoid, we obtain the following theorem, which asserts that every concentration arises as a pullback.

\begin{theorem} \label{thm: Every concentration structure is pullback}
    Every concentration structure is a pullback of a discrete concentration structure.
\end{theorem}

In general, any monoid can be regarded as a one-object category with discrete concentration structure, which can be viewed as a functor $\mathbf{C_{\sim}}$ from $\mathscr{M}on$ to $\mathscr{C}at_{\sim} $ (see Section \ref{subsec: Concentration monoid and adjoint functors} for the precise description). It turns out that  $\mathbf{C_\sim}$ and the concentration monoid functor $\mathbf{M}$ form an adjoint pair.   

\begin{theorem} \label{thm: M is left adjoint to C}
    $\mathbf{M}$ is left adjoint to $\mathbf{C_\sim}$.
\end{theorem}


We next investigate how $\mathbf{M}$ behaves with respect to standard algebraic constructions. In analogy with the familiar sub-, quotient-, and semidirect-product operations for monoids, we define corresponding notions for categories with concentration. The following theorem shows that these constructions are preserved by $\mathbf{M}$, as expected.


\begin{theorem} \label{thm: concentration restrict to sub,quotient, and semidirect product}
    The functor $\mathbf{M}$ preserves sub-, quotient-, and semidirect product structures from categories with concentration to monoids.
\end{theorem}

As stated above, Theorem \ref{thm: concentration restrict to sub,quotient, and semidirect product} contains three cases: sub-concentrations, quotient concentrations, and semidirect product. They correspond respectively to Proposition \ref{prop: stabilizatin of sub concentration is submonoid}, Proposition \ref{prop: concentration monoid of quotient concentration is the quotient of monoid}, and Proposition \ref{prop: concentration of semidirect product is semidirect product of concentration}.

\subsection{Horizontal categorification and decategorification}

With concentration structures and concentration monoids, We define horizontal categorification as follows.

\begin{definition}[\textbf{Horizontal categorification}]
    We say a category $\C$ is an (internal) horizontal categorification of a monoid $M$ if there exists a concentration structure $\sim$ on $\C$ such that $M_{(\C, \sim)} \cong M$.
\end{definition}

A more intuitive definition of horizontal decategorification comes from the definition of $2$-lifting functor, which guarantees that the source category contains the entire $2$-composition structure in the target category.

\begin{definition} \label{def: $2$-lifting functor}
    A functor $F: \mathcal{C} \to \mathcal{D}$  is a \textbf{$2$-lifting functor} if for any morphisms $g_1, g_2$ in $\mathcal{D}$ such that $g_1 \circ g_2 $ exists, there exist some morphisms $f_1, f_2$ in $\mathcal{C}$, satisfying $f_1 \circ f_2 $ exists and $F(f_1) = g_1, F(f_2) = g_2$.
\end{definition}
Given a monoid $M$, let $\M$ be the canonical category consisting of a single object $*$ and $\textup{Mor}(*, *) = M$, where the composition of morphisms is given by the multiplication in $M$.
\begin{definition}[\textbf{External horizontal categorification}]
    We say a category $\C$ is an external horizontal categorification of a monoid $M$ if there exists a $2$-lifting functor from $\C$ to $\M$.
\end{definition}

\begin{theorem} \label{thm: internal and external categorifications are equivalent}
    The above two definitions of horizontal categorification are equivalent.
\end{theorem}

We will prove the above theorem in Section \ref{subsec: decategorification} by considering pullback of $2$-lifting functors. The definition of horizontal decategorification is formulated in terms of horizontal categorification. 

\begin{definition}[\textbf{Horizontal decategorification}]
    We say a monoid $M$ is a horizontal decategorification of a category $\C$ if $\C$ is a horizontal categorification of $M$.
\end{definition}

\subsection{$G$-equivariant direct limits}

Given a directed set $S$, a direct system on $S$ is a functor $F$ from the associated direct category $\mathcal{S}$ to the category of groups $\mathscr{G}rp$ (see Section \ref{sec: Concentrations and $G$-equivariant direct limits}). The colimit of the functor $F$ is also called the direct limit of the direct system $F$, denoted as $\displaystyle\lim_{\longrightarrow} F$.

Using concentration, we are able to reinterpret  $\displaystyle\lim_{\longrightarrow} F$ as the concentration group of a certain groupoid with concentration $(\mathcal{S}_0, \sim)$. It turns out that when $\mathcal{S}$ has a $G$-action on it for some group $G$, and the functor $F$ is $G$-equivariant, we can construct a groupoid with concentration $(\mathcal{S}_G, \sim)$ as a generalization of $(\mathcal{S}_0, \sim)$. 


Thus, the direct limit of $F$ can be naturally generalized to the concentration monoid of $(\mathcal{S}_G, \sim)$, which we call the $G$-equivariant direct limit of $F$, and denote by $\displaystyle \lim_{\longrightarrow}{}^GF$. Moreover, $\displaystyle \lim_{\longrightarrow}{}^GF$ is related to $\displaystyle\lim_{\longrightarrow} F$ by the following theorem. 

\begin{theorem}
    $\displaystyle \lim_{\longrightarrow}{}^GF$ is isomorphic to a semidirect product of $\displaystyle\lim_{\longrightarrow} F$ and $G$.
\end{theorem}

As an example of $G$-equivariant direct limit, in Section \ref{sec: real braid groups} we construct several versions of ``$\mathbb{R}$-braid group" via group actions on $\mathbb{R}$. In particular, they extend the infinite braid group \cite{MR19087} \cite{MR19088} on the integers to a continuous index set.


\subsection{Concentration structures on fundamental groupoids}

An interesting class of examples of concentration structure arises from the fundamental groupoid $\Pi(X)$ of a path-connected topological space $X$. Moreover, the concentration group of such a concentration structure can recover the fundamental group of $X$.

\begin{theorem}\label{thm: concentration monoid of fundamental groupoid is fundamental group}
    For any path-connected topological space $X$, there exists some concentration structure $\sim$ on the fundamental groupoid $\Pi(X)$, such that the concentration group $M_{(\Pi(X),\sim)}\cong \pi_1(X)$. In other words, $\Pi(X)$ is a horizontal categorification of $\pi_1(X)$.
\end{theorem}

Combine the above theorem and the pullback along universal covering map, we are able to conclude the following theorem, which demonstrates that the concentration structures can vary significantly under categorical equivalence.

\begin{theorem} \label{thm: any group can be recovered as concentration group of trivial category}
    For any group $G$, there exists some category with concentration $(\C,\sim)$ such that $\C$ is equivalent to the trivial category, and the concentration monoid $M_{(\C,\sim)}$ is isomorphic to $G$.
\end{theorem}

\subsection{Organization}

The present paper is organized as follows. In Section \ref{sec: general definitions} we develop the general theory of categories with concentration, including pullbacks, concentration monoids, horizontal categorification, and sub- and quotient concentrations. Section \ref{sec: Concentrations and semidirect products} introduces the semidirect products of categories with concentration. In Section \ref{sec: Concentrations and $G$-equivariant direct limits} we define the $G$-equivariant direct limit and study its relationship with semidirect products. Finally, Section \ref{sec: concentrations fundamental groupoids and fibrations} investigates concentration structures on fundamental groupoids and discusses their applications.

\subsection*{Acknowledgement}
The authors would like to thank Slava Krushkal and Brandon Shapiro for useful suggestions. 
Yangxiao Luo was supported in part by NSF grant DMS-2105467 to Slava Krushkal. 


\section{General properties of categories with concentration}\label{sec: general definitions}

In this section we develop the basic formalism of concentration structures on categories. We begin by proving the well-definedness and functoriality of the concentration monoid construction. We then introduce the pullback of concentration structures along 
$2$-lifting functors and establish its basic properties. In Section \ref{subsec: Concentration monoid and adjoint functors} and \ref{subsec: sub and quotient concentrations}, we turn to a closer study of the concentration monoid functor, including the prove of Theorem \ref{thm: M is left adjoint to C} and part of Theorem \ref{thm: concentration restrict to sub,quotient, and semidirect product}. These results lay the groundwork for the constructions and applications developed in the later sections.

\subsection{Well-definedness and functoriality of taking concentration monoid}
\label{sec: well-definedness and functoriality of taking concentration monoid}

Recall that for a category with concentration $(C,\sim)$, the concentration monoid $M_{(C,\sim)}$ is defined as the quotient $\mathrm{Mor}_{\C} / \sim$ with multiplication
$$
[f]\,[g] \;=\; [\,f' \circ g'\,]
$$
where $f' \sim f$ and $g' \sim g$ are chosen so that $f' \circ g'$ exists.

In the next proposition we will verify that this operation is well defined and $M_{(C,\sim)}$ is indeed a monoid, which is the first step towards the functoriality of taking concentration monoid (Theorem \ref{thm: Taking concentration monoid is a functor}).

\begin{proposition} 
\label{prop: G is a monoid}
    The multiplication on $M_{(\mathcal{C}, \sim)}$ is well defined, and $M_{(\mathcal{C}, \sim)}$ is a monoid. Moreover, if any morphism in $\mathcal{C}$ is equivalent to some isomorphism, then $M_{(\mathcal{C}, \sim)}$ is a group, called the concentration group of $(\mathcal{C}, \sim)$. In particular, $M_{(\mathcal{C}, \sim)}$ is a group whenever $\C$ is a groupoid.
\end{proposition}

\begin{proof}
    We first check the well-definedness of the multiplication. Suppose that $f_1 \sim f_2$ and $g_1 \sim g_2$. By the definition of the multiplication in $M_{(\mathcal{C}, \sim)}$, we have $[f_1][g_1] = f_1' \circ g_1'$ and $[f_2][g_2] = f_2' \circ g_2'$, for some $f_1' \sim f_1, g_1' \sim g_1, f_2' \sim f_2, g_2' \sim g_2$ such that $f_1' \circ g_1'$ and $f_2' \circ g_2'$ exist. Note that $f_1' \sim f_1 \sim f_2 \sim f_2'$ and $g_1' \sim g_1 \sim g_2 \sim g_2'$, then by the second axiom of concentration, we have $f_1' \circ g_1' \sim f_2' \circ g_2'$ which means $[f_1][g_1]=[f_2][g_2]$. 

    To check that $M_{(\mathcal{C}, \sim)}$ is a monoid, we need to show the existence of identity and the associativity of the multiplication. First, the identity element is given by $[id_A]$, the class of the identity morphism from arbitrary object $A \in \mathcal{C}$ to itself. The first axiom of concentration guarantees its well-definedness. For any $[f] \in M_{(\mathcal{C}, \sim)}$, suppose that $f \in \textup{Mor}(B, C)$, then $[f][id_A] = [f \circ id_B] = [f]$ and $[id_A][f] = [id_C \circ f] = [f]$. Thus $[id_A]$ is indeed an identity.


    For any $[f], [g], [h] \in M_{(\mathcal{C}, \sim)}$, by the third axiom we can always find $f' \sim f, g' \sim g\sim g'',  h\sim h''$ such that $f' \circ g'$ and $g'' \circ h''$ exist, similarly we can also find $f'' \sim f', n \sim g'' \circ h'', m\sim f'\circ g', h'\sim h''$ such that $f'' \circ n$ and $m \circ h'$ exist, and then by the fourth axiom $f'' \circ n \sim m \circ h'$. Thus $[f]([g][h])=[f''\circ n]=[m \circ h']=([f][g])[h]$, so the associativity holds.

    If any morphism $f$ in $\mathcal{C}$ is equivalent to some isomorphism $f'$, then we claim that $[{f'}^{-1}]$ is a well-defined inverse of $[f]$. Suppose that $f$ is also equivalent to some isomorphism $f''$, then $[{f''}^{-1}] = [{f'}^{-1}][f'][{f''}^{-1}] =[{f'}^{-1}][f''][{f''}^{-1}] = [{f'}^{-1}]$. Additionally, $[f][{f'}^{-1}] = [{f'}][{f'}^{-1}] = id$ and $[{f'}^{-1}][f] = [{f'}^{-1}][f'] = id$. Thus $[{f'}^{-1}]$ is a well-defined inverse of $[f]$, and $M_{(\mathcal{C}, \sim)}$ is a group.
\end{proof}


We now turn to the second step in the proof of Theorem \ref{thm: Taking concentration monoid is a functor}. Having established that the concentration monoid is a well-defined monoid, we next study how concentration preserving functors interact with this construction. The following proposition shows that any such functor naturally induces a homomorphism between the associated concentration monoids.

\begin{proposition} \label{prop: concentration preserving functor induces monoid homomorphism}
    Any concentration preserving functor $F: (\C,\sim_{\C}) \rightarrow(\D,\sim_{\D})$ induces a monoid homomorphism $\phi_F: M_{(\C,\sim_{\C})} \to M_{(\D,\sim_{\D})}$ where $\phi_F([f])=[F(f)]$. Moreover, if $F$ is a concentration isomorphism, then $\phi_F$ is a monoid isomorphism.
\end{proposition}

\begin{proof}
    To show $\phi_F$ is a monoid homomorphism we need to check it preserves identity and multiplication. By Proposition \ref{prop: G is a monoid} we know the identity in $M_{(\C,\sim_{\C})}$ is the class of the identity morphisms in $\C$. We have
    $$
    \phi_F([id_A])=[F(id_A)]= [id_{F(A)}].
    $$
    
    Thus $\phi_F$ preserves identity. Next we check that $\phi_F$ preserves multiplication. For any $[f], [g] \in M_{(\C,\sim_{\C})}$, find $f' \sim f$ and $g' \sim g$ such that $f' \sim g'$ exists, then
    \begin{align*}
        \phi_F([f][g])&=\phi_F([f'\circ g'])=[F(f'\circ g')]=[F(f')\circ F(g')] =[F(f')][F(g')]\\
        &=[F(f)][F(g)] 
        \quad \text{(since $F$ is concentration preserving )}\\
        &= \phi_F([f])\phi_F([g])
    \end{align*}
    which shows that $\phi_F$ is a homomorphism. 

    When $F$ is a concentration isomorphism (Definition \ref{def: concentration preserving functor}) with strong inverse $F^{-1}$, it is easy to see that $\phi_{F^{-1}}$ is the inverse of $\phi_{F}$, so $\phi_F$ is an isomorphism. 
\end{proof}


Next, we formalize the process of taking concentration monoids as a functor from $\mathscr{C}at_{\sim}$ to $\mathscr{M}on$. The two propositions above ensure that this construction is well defined on both objects and morphisms.

\begin{definition} \label{def: the functor taking concentration monoid}
    Define the functor $\mathbf{M}: \mathscr{C}at_{\sim} \to \mathscr{M}on$ as follows. For objects, $\mathbf{M}$ sends a category with concentration $(\mathcal{C}, \sim_{\mathcal{C}})$ to the concentration monoid $M_{(\mathcal{C}, \sim_{\mathcal{C}})}$. For morphisms, $\mathbf{M}$ sends a concentration preserving functor $F:(\C,\sim_{\C}) \rightarrow(\D,\sim_{\D})$ to the induced monoid homomorphism $\phi_F: M_{(\C,\sim_{\C})} \to M_{(\D,\sim_{\D})}$. We call $\mathbf{M}$ the \textbf{concentration monoid functor}.
\end{definition}

We are ready to prove Theorem \ref{thm: Taking concentration monoid is a functor}.

\begin{proof} [Proof of Theorem \ref{thm: Taking concentration monoid is a functor}]
    The well-definedness of $\mathbf{M}$ follows from Proposition \ref{prop: G is a monoid} and Proposition \ref{prop: concentration preserving functor induces monoid homomorphism}. Next we show that $\mathbf{M}$ preserves identity morphisms and compositions.

    Let $id_{\mathcal{C}}$ be the identity functor (which is obviously a concentration preserving functor) from a category with concentration $(\mathcal{C}, \sim_{\mathcal{C}})$ to itself. Its induced monoid homomorphism $\phi_{id_\mathcal{C}}$ sends $[f] \in M_{(\mathcal{C}, \sim_{\mathcal{C}})}$ to $[id_\mathcal{C}(f)] = [f]$, so $\phi_{id_\mathcal{C}}$ is an identity map.

    Given two concentration preserving functors $F: (\mathcal{C}, \sim_{\mathcal{C}}) \to (\mathcal{D}, \sim_{\mathcal{D}})$ and $G: (\mathcal{D}, \sim_{\mathcal{D}}) \to (\mathcal{E}, \sim_{\mathcal{E}})$. Their composition $G \circ F$ induces a monoid homomorphism $\phi_{G \circ F}: M_{(\mathcal{C}, \sim_{\mathcal{C}})} \to M_{(\mathcal{E}, \sim_{\mathcal{E}})}$, sending $[f]$ to $[G \circ F(f)] = \phi_G[F(f)] = \phi_G \circ \phi_F[f]$. Thus, $\phi_{G \circ F} = \phi_G \circ \phi_F$, meaning $\mathbf{M}$ preserves compositions.

    By Proposition \ref{prop: G is a monoid}, the functor $\mathbf{M}$ sends groupoids with concentration to groups, so it can also be regarded as a functor from $\mathscr{G}rpd_{\sim}$ to $\mathscr{G}rp$.
\end{proof}

\begin{remark}
    Note that a horizontal decategorification of a groupoid must be a group, but it is possible that a horizontal categorification of a group is a category with non-invertible morphisms.
\end{remark}



\subsection{$n$-concentration structures}\label{subsec:n-concentration}
The fourth condition of the concentration structure (Definition \ref{def: concentration structure}) is often difficult to verify. Instead, we will introduce more general $n$-concentration structures that help us check when an equivalence relation is a concentration structure. 
\begin{definition}\label{def: n-concentration}
    An \textbf{$n$-concentration structure} on $\mathcal{C}$ is an equivalence relation $\sim$ on $\textup{Mor}_{\mathcal{C}}$, which satisfies
    
    \begin{enumerate}
        \item $id_A \sim id_B$ for any objects $A, B \in \textup{Ob}_{\mathcal{C}}$. (Identity axiom)
        
        \item If $f \sim f', g \sim g'$ and both $f\circ g, f' \circ g'$ exist, then $f \circ g \sim f' \circ g'$. (Composition axiom)
\item For any morphisms $f_1, f_2,\dots,f_n \in \textup{Mor}_{\mathcal{C}}$, there exist some $f_1' \sim f_1, f_2' \sim f_2,\dots,f_n'\sim f_n$ such that $f_1' \circ f_2'\circ \cdots \circ f_n'$ exists. ($n$-Existence axiom)

        \end{enumerate} 

\end{definition}

The actual concentration structure can be viewed as a $2$-concentration structure with the fourth axiom (Associativity axiom). 
It is clear that an $(n+1)$-concentration structure implies an $n$-concentration structure. More importantly, a $3$-concentration structure implies not only a $2$-concentration structure but also the actual  concentration structure.

\begin{proposition} \label{prop: 3-concentration implies concentration}
    Let $\sim$ on $\C$ be a $3$-concentration structure, then $\sim$ is actually a concentration structure on $\C$.
\end{proposition}

\begin{proof}
   We only need to check the associativity axiom. Given morphisms $f\sim f', g\sim g', h\sim h', m\sim f\circ g, n \sim g'\circ h'$ with $f' \circ n, m \circ h$ exist, we want to show $f' \circ n \sim m \circ h$. The $3$-existence axiom implies there exist $f''\sim f, g''\sim g,h''\sim h$ such that $f''\circ g'' \circ h''$ exist. Then the composition axiom implies $f'\circ n \sim f'' \circ g''\circ h'' \sim m \circ h$, which completes the proof.
\end{proof}

Thus, we can think of the concentration structure as a ``2.5-concentration structure" sitting between $2$-, and $3$-concentration structures. Moreover, the axioms of $3$-concentration structures usually provide an easier way to verify when an equivalence relation is a concentration structure.

\begin{example} \label{ex: Z/2 and Z/4 category with concentration}
We consider a category $\C$ with two objects $C$ and $D$, where $\textup{Mor}_\C(C, C) = \mathbb{Z}/2$ with morphisms denoted as $0_C$ and $1_C$, and $\textup{Mor}_\C(D, D) = \mathbb{Z}/4$ with morphisms denoted as $0_D$, $1_D$, $2_D$ and $3_D$. There are no morphisms between $C$ and $D$. See Figure \ref{fig: Z/2 and Z/4 example} for an illustration. We can construct some non-trivial $3$-concentration structures (in particular concentration structures) on $\C$ as follows. 

\begin{enumerate}
    \item The first concentration structure $\sim_a$ is given by $0_C\sim_a0_D$, $1_C\sim_a 2_D$. The corresponding concentration monoid is $M_{(\C,\sim_a)}\cong \mathbb{Z}/4$
    \item  The second concentration structure $\sim_b$ is given by $0_C\sim_b0_D\sim_b2_D$, $1_C\sim_b 1_D\sim_b3_D$. The corresponding concentration monoid is $M_{(\C,\sim_b)}\cong \mathbb{Z}/2$.
    \item The third concentration structure $\sim_c$ is given by $0_C\sim_c 1_C \sim_c 0_D$. The corresponding concentration monoid is $M_{(\C,\sim_c)}\cong \mathbb{Z}/4$
    
    \item The fourth concentration structure $\sim_d$ is given by $0_C\sim_d 0_D\sim_d 1_D\sim_d 2_D\sim_d 3_D$. The corresponding concentration monoid is $M_{(\C,\sim_d)}\cong \mathbb{Z}/2$
\end{enumerate}
\end{example}
These examples show that two concentration structures can be different while their concentration monoids are isomorphic. Moreover, the above examples can be viewed as a concentration interpretation of the direct limit of groups, which we will discuss more in Section \ref{sec: Concentrations and $G$-equivariant direct limits}.

\begin{figure}[h]

\[\begin{tikzcd}
	C && D
	\arrow["{\mathbb{Z}/2}", from=1-1, to=1-1, loop, in=55, out=125, distance=10mm]
	\arrow["{\mathbb{Z}/4}", from=1-3, to=1-3, loop, in=55, out=125, distance=10mm]
\end{tikzcd}\]
\caption{The category in Example \ref{ex: Z/2 and Z/4 category with concentration}.}
\label{fig: Z/2 and Z/4 example}
\end{figure}
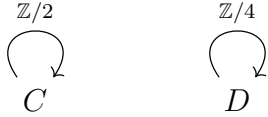

\subsection{Pullbacks of concentration structures}


Recall that a $2$-lifting functor (Definition \ref{def: $2$-lifting functor}) is a functor 
preserving the existence of compositions when morphisms are lifted from the target category back to the source category. This property turns out to be exactly what is needed to pull back concentration structures, and makes $2$-lifting functors a useful tool for generating new concentration structures from existing ones.


It is helpful to look at some concrete examples of $2$-lifting functors. In particular, the second example will reappear in Section \ref{sec: fibrations and concentration structures on fundamental groupoids}.

\begin{example} Let $F: \mathcal{C} \to \mathcal{D}$ be a functor.
    \begin{enumerate}
        \item If $F$ has a right inverse $G$, i.e. $F \circ G = id$, then $F$ is $2$-lifting since we can let $f_i = G(g_i)$.

        \item If $F$ is a surjective Grothendieck fibration, or more generally, a surjective multivalued fibration (Definition \ref{def: multivalued fibration}), then $F$ is $2$-lifting (Proposition \ref{prop: fibration induces $2$-lifting functor}). 
    \end{enumerate}
     
\end{example}

Given a $2$-lifting functor $F : \C \to \D$ and a concentration structure on $\D$, we can pull back the concentration to $\C$, as described in the following definition. Moreover,  Lemma \ref{prop: pullback concentration of $2$-lifting functor} confirms that this pullback satisfies all four axioms of concentration structure.

\begin{definition}
    Let $F: \mathcal{C} \to \mathcal{D}$ be a $2$-lifting functor and $\sim_{\D}$ be a concentration structure on $\D$. We define a concentration structure $\sim_{\D}^{F^*}$ on $\C$, such that $f\sim_{\D}^{F^*} g$ if and only if $F(f) \sim_\D F(g)$. We call it the pullback concentration structure of $\sim_\D$ along $F$, or simply the \textbf{pullback} of $\sim_\D$.
\end{definition}

\begin{lemma}\label{prop: pullback concentration of $2$-lifting functor}
    The pullback $\sim_{\D}^{F^*}$ is a concentration structure on $\mathcal{C}$.
\end{lemma}

\begin{proof}
    It is easy to check that $\sim_{\D}^{F^*}$ is an equivalence relation. Next we check the four axioms of concentration in Definition \ref{def: concentration structure}. 

    \textbf{Axiom 1.} Let $A,B$ be two objects in $\C$, and $id_{A},id_{B}$ be their identity morphisms. Then $F(id_A)=id_{F(A)}\sim_\D id_{F(B)}=F(id_B)$, and thus $id_A \sim_\D^{F^*} id_B$.

    \textbf{Axiom 2.} Suppose that $f \sim_{\D}^{F^*} f', g \sim_{\D}^{F^*} g'$, and both of $f \circ g, f' \circ g'$ exist. Then $F(f) \sim_{\D} F(f'), F(g) \sim_{\D} F(g')$ by the definition of pullback. Since $\sim_\D$ is a concentration structure, we have $F(f \circ g) = F(f) \circ F(g) \sim_{\D} F(f') \circ F(g') = F(f' \circ g')$. Thus $(f \circ g) \sim_\D^{F*} (f' \circ g')$.

    \textbf{Axiom 3.} Let $f,g$ be two morphisms in $\C$, pick two morphisms $\bar{f},\bar{g}$ in $\D$ such that $\bar{f}\sim_\D F(f),\bar{g}\sim_\D F(g)$ and $\bar{f}\circ\bar{g}$ exist. Since $F$ is $2$-lifting, we can find $f',g'$ in $\C$ such that $f'\circ g'$ exist, and $F(f')=\bar{f},F(g')=\bar{g}$. Then $F(f')\sim_\D F(f),F(g')\sim_\D F(g)$, and thus $f' \sim_\D^{F*} f$ and $g'\sim_\D^{F*}g$.

    \textbf{Axiom 4.} Given any morphisms $f\sim_\D^{F*} f', g\sim_\D^{F*} g', h\sim_\D^{F*} h', m\sim_\D^{F*} f\circ g, n \sim_\D^{F*} g'\circ h'$ with $f' \circ n, m \circ h$ exist,  we want to show $f' \circ n \sim_\D^{F*} m \circ h$. By the definition of pullback, we have $F(f)\sim_{\D} F(f'), \ F(g)\sim_\D F(g'), \ F(h)\sim_\D F(h'), \ F(m)\sim_\D F(f)\circ F(g),\  F(n) \sim_\D F(g')\circ F(h')$ with $F(f') \circ F(n),\  F(m) \circ F(h)$ exist. Since $\sim_D$ satisfies the associativity axiom, we have $F(f') \circ F(n) \sim_D F(m) \circ F(h)$. Then $F(f' \circ n) \sim_D F(m \circ h)$, which implies $f' \circ n \sim_\D^{F*} m \circ h$.
\end{proof}


Lemma \ref{prop: pullback concentration of $2$-lifting functor} shows that the pullback along a $2$-lifting functor yields a new concentration structure from an existing one, we now compare their concentration monoids. The next lemma shows that such a functor induces an isomorphism between two concentration monoids.

\begin{lemma} \label{lem: pullback concentration induce isomorphism on concentration monoids}
    Let $F: \C \to \D$ be a $2$-lifting functor, and $\sim_\D$ be a concentration structure on $\D$. Then
    $F: (\C,\sim_{\D}^{F^*}) \to (\D,\sim_\D)$ is a concentration preserving functor, and it induces an isomorphism between the concentration monoids $M_{(\C,\sim_{\D}^{F^*})}$ and $M_{(\D,\sim_\D)}$.
\end{lemma}

\begin{proof}
    If $f \sim_{\D}^{F*} g$, then by the definition of pullback we have $F(f) \sim_\D F(g)$, which exactly means that $F$ is a concentration preserving functor.

    Next we show that the induced monoid homomorphism $\phi_F: M_{(\C,\sim_{\D}^{F^*})} \to M_{(\D,\sim_\D)}$ is an isomorphism. Note that any $2$-lifting functor must be surjective, so for any $[g] \in M_{(\D,\sim_\D)}$, there exists a morphism $f$ in $\C$ such that $F(f) = g$. Then $\phi_F[f] = [g]$, which means $\phi_F$ is surjective.
    
    Suppose that $\phi_F[f] = \phi_F[f']$, then $F(f) \sim_\D F(f')$. By the definition of pullback, we have $f \sim_\D^{F*} f'$, so $\phi_F$ is injective.
\end{proof}

Next we will use this fact that pullbacks along $2$-lifting functors preserve the induced concentration monoids to prove Theorem \ref{thm: Every concentration structure is pullback}.

\vspace{5mm}

Given a category with concentration $(\C, \sim)$, we take its concentration monoid $M_{(\C, \sim)}$. Let $\M_{(\mathcal{C}, \sim)}$ be the category containing single object $*$ and $\textup{Mor}(*, *) = M_{(\mathcal{C}, \sim)}$, where the composition of morphisms is given by the multiplication in $M_{(\mathcal{C}, \sim)}$. Then there is a natural concentration preserving functor
\begin{equation}\label{equ: concentrating functor}
    F_{(C,\sim)}: (\C,\sim) \rightarrow(\M_{(\C,\sim)},\sim_{dis})
\end{equation}
sending any object in $\C$ to $*$, and morphism $f$ to $[f]$. We call $F_{(C,\sim)}$ the \textbf{concentrating functor} associated to $(\C,\sim)$.

Next lemma shows that the concentrating functor is $2$-lifting.

\begin{lemma} \label{lem: concentrating functor is 2-lifting}
    $F_{(C,\sim)}$ is $2$-lifting, as a functor from $\C$ to $\M_{(\C,\sim)}$.
\end{lemma}

\begin{proof}
    For any morphisms $[f], [g]$ in $\M_{(\C,\sim)}$, note that $f, g$ are morphisms in $\C$, so the third axiom of concentration guarantees that there exist $f' \sim f, g' \sim g$ such that $f' \circ g'$ exists. Moreover, $F_{(\C,\sim)}(f') = [f'] = [f]$, and similarly, $F_{(\C,\sim)}(g') = [g]$. Thus $F_{(C,\sim)}$ is $2$-lifting.
\end{proof}


If we take the pullback of the discrete concentration structure $\sim_{dis}$ on $\M_{(C,\sim)}$ along $F_{(C,\sim)}$, we obtain a new concentration structure $\sim^{F^*_{(C,\sim)}}_{dis}$ on $\C$. The following proposition confirms that this pullback coincides exactly with the original concentration $\sim$.

\begin{proposition}
    For any category with concentration $(\C,\sim)$ . The pullback concentration $\sim^{F^*_{(C,\sim)}}_{dis}$ is the same as the original concentration $\sim$ on $\C$.
\end{proposition}

\begin{proof}
    For simplicity we denote $\sim^{F^*_{(C,\sim)}}_{dis}$ as $\sim'$. We need to show for any two morphisms $f,g$ in $\C$, $f\sim g\iff f\sim'g$.

    Note that $F_{(\C, \sim)}(f) = [f]_{\sim}$ and $F_{(\C, \sim)}(g) = [g]_{\sim}$, so $[f]_{\sim} = [g]_{\sim}$ if and only if $F_{(\C, \sim)}(f) \sim_{dis} F_{(\C, \sim)}(g)$. In other words, $f \sim g$ if and only if $f \sim' g$.
\end{proof}

This proposition shows that every concentration structure arises as the pullback of a discrete one, thereby giving a proof of Theorem \ref{thm: Every concentration structure is pullback}. 

\subsection{$2$-lifting functors and horizontal decategorification}\label{subsec: decategorification}

Note that any $2$-lifting functor (Definition \ref{def: $2$-lifting functor}) is necessarily surjective. However, we don't want to impose surjectivity alone in the definition of horizontal decategorification. The extra $2$-lifting property ensures the source category contains the entire multiplication structure of the monoid. The following example illustrates a surjective functor that is not $2$-lifting. We will see why the source category is not an appropriate horizontal categorification of the corresponding monoid. 

\begin{example} \label{ex: surjective but not 2-lifting functor}
We consider a category $\C$ with three objects $C$ $D$ and $E$, and three non-trivial morphisms $f,g,h$, where $f: C \to D$, $g:D \to E$ and $h=g \circ f: C \to E$.  See Figure \ref{fig: surjective but not 2-lifting functor} for an illustration.    

We construct a surjective functor $F$ from $\C$ to $\D$, where $\D$ is the category  with one object representing $\mathbb{Z}/2$, as follows. $F$ maps all objects $C,D,E$ in $\C$ to the single object in $\D$, and sends the morphisms $f$ to $0$, $g$ to $1$, and $h$ to $1$. 

It is straightforward to check that $F$ is a surjective functor, but $\C$ is not an appropriate horizontal categorification of $\mathbb{Z}/2$, since the cyclic structure in $\mathbb{Z}/2$ is not reflected in $\C$. For a legit horizontal categorification of cyclic group, see Example \ref{ex: not 3-concentration} and Figure \ref{fig: z/3}.
\end{example}

The $2$-lifting definition of horizontal categorification and decategorification is intuitive, but not intrinsic to the category itself. Constructing a $2$-lifting functor from a category $\C$ to a single-object category is generally nontrivial because finding the appropriate monoid is hard. In practice, 
such a monoid always comes from a suitable concentration structures on $\C$.

\begin{figure}[H]
\[\begin{tikzcd}
	C & D & E
	\arrow["f", curve={height=-6pt}, from=1-1, to=1-2]
	\arrow["h"', curve={height=12pt}, from=1-1, to=1-3]
	\arrow["g", curve={height=-6pt}, from=1-2, to=1-3]
\end{tikzcd}\]
\caption{The category in Example \ref{ex: surjective but not 2-lifting functor}.}
\label{fig: surjective but not 2-lifting functor}
\end{figure}

\begin{proof}[Proof of Theorem \ref{thm: internal and external categorifications are equivalent}]

We first show the case for category and monoid, then the case for groupoid and group follows from Theorem \ref{thm: Taking concentration monoid is a functor}. Let $\C$ be a category and let $M$ be a monoid. Suppose that $\C$ is an internal horizontal categorification of $M$, then there exists a concentration structure $\sim$ on $\C$ such that $M_{(\C, \sim)} \cong M$. By Lemma \ref{lem: concentrating functor is 2-lifting}, the concentrating functor $F_{(\C, \sim)}: \C \to \M_{(\C, \sim)}$ is a $2$-lifting functor. Compose with the functor $\M_{(\C, \sim)} \to \M$ induced by the isomorphism $M_{(\C, \sim)} \to M$, we obtain a $2$-lifting functor $\C \to \M$, which means that $\C$ is an external horizontal categorification of $M$.

Suppose that $\C$ is an external horizontal categorification of $M$, then there exists a $2$-lifting functor $F: \C \to \M$. Consider $\sim_{dis}^{F^*}$, the pullback of the discrete concentration structure on $\M$, Lemma \ref{lem: pullback concentration induce isomorphism on concentration monoids} tells us that $M_{(\C, \sim_{dis}^{F^*})} \cong M_{(\M, \sim_{dis})}$. Thus we have $M_{(\C, \sim_{dis}^{F^*})} \cong M$, meaning $\C$ is an internal horizontal categorification of $M$.
\end{proof}

\begin{remark}
    Following the observation that a $2$-lifting functor from a category to a one-object category can lift the multiplication to the composition, we can potentially define the horizontal categorification of an algebraic object to be a multi-object category admitting a functor lifting the entire algebraic structure.
\end{remark}

\subsection{Concentration monoid and adjoint functors}\label{subsec: Concentration monoid and adjoint functors}

In this section we prove Theorem \ref{thm: M is left adjoint to C}, which asserts that the concentration monoid functor $\mathbf{M}$ is left adjoint to the functor $\mathbf{C}_{\sim}$, which sends a monoid to the associated category with discrete concentration. We first give a detailed description of $\mathbf{C_{\sim}}$.

Recall that given a monoid $M$, we denote $\M$ to be the category consisting of a single object $*$ and $\textup{Mor}(*, *) = M$, where the composition of morphisms is given by the multiplication in $M$. This construction can be encoded as a functor $\mathbf{C}: \mathscr{M}on \to \mathscr{C}at$. It sends a monoid homomorphism $\phi: M \to N$ to a functor $\Phi: \M \to \mathcal{N}$ such that $\Phi(*) = *$ and $\Phi(x) = \phi(x)$ for $x \in M$. The functor
\begin{equation} \label{equ: the functor C}
    \mathbf{C_{\sim}}: \mathscr{M}on \to \mathscr{C}at_{\sim}
\end{equation}
is defined analogously to $\mathbf{C}$, sending $M$ to $(\mathcal{M},\sim_{dis})$ and $\phi$ to $\Phi$.

To establish the adjunction between $\mathbf{M}$ and $\mathbf{C}_{\sim}$, we adopt the unit–counit definition of adjunction.




\begin{definition}
    Let $\mathcal{C}, \mathcal{D}$ be two categories, and let $F: \mathcal{C} \to \mathcal{D}$ and $G: \mathcal{D} \to \mathcal{C}$ be a pair of functors. We say $F$ is \textbf{left adjoint} to $G$, if there exist natural transformations $\epsilon: FG \to 1_{\C}$ and $\eta: 1_{\mathcal{D}} \to GF$ such that the compositions
    \begin{align*}
        F \xrightarrow{F \eta} FGF \xrightarrow{\epsilon F} F \\
        G \xrightarrow{\eta G} GFG \xrightarrow{G \epsilon} G
    \end{align*}
    are identities. $\epsilon$ is called the counit of the adjunction, $\eta$ is called the unit of the adjunction.
\end{definition}

\begin{proof}[Proof of Theorem \ref{thm: M is left adjoint to C}]

We want to show that $\mathbf{M}: \mathscr{C}at_{\sim} \to \mathscr{M}on$ (Definition \ref{def: the functor taking concentration monoid}) is left adjoint to $\mathbf{C}_\sim:\mathscr{M}on \to \mathscr{C}at_{\sim}$ (Eq. \ref{equ: the functor C}), and we will prove it by constructing a counit $\epsilon: \mathbf{MC_\sim} \to 1_{\mathscr{M}on}$ and a unit $\eta: 1_{\mathscr{C}at_{\sim}} \to \mathbf{{C_\sim}M}$.


Let $M$ be a monoid, then $\mathbf{C}_\sim(M) = (\mathcal{M}, \sim_{dis})$, the canonical category associated to $M$ with discrete concentration. We identify $M$ with the concentration monoid $M_{(\mathcal{M}, \sim_{dis})}$, by identifying $f\in M$ with $[f]_{\sim_{ dis}} \in M_{(\mathcal{M}, \sim_{dis})}$. Then we have $\mathbf{MC_\sim}(M) = M$.

Let $(\C, \sim)$ be a category with concentration, then $\mathbf{C_\sim{M}}(\C,\sim)=(\M_{(\C,\sim)},\sim_{dis})$, the canonical one-object category associated to the concentration monoid $M_{(\C,\sim)}$ with discrete concentration.


Define the counit $\epsilon = \{\epsilon_M \in \textup{Mor}(M, M)\}$ and the unit $\eta = \{\eta_{(\C, \sim)}\in \textup{Mor}((\C, \sim), (\mathcal{M}_{(\C,\sim)}, \sim_{dis}))\}$ such that
\begin{align*}
    &\epsilon_M = id_M \\
    &\eta_{(\C, \sim)} = F_{(\C, \sim)}
\end{align*}
where $F_{(\C, \sim)}$ is the concentrating functor (Eq. \ref{equ: concentrating functor}) associated to $(\C, \sim)$.

Next we check that
\begin{align*}
    & \mathbf{M} \xrightarrow{\mathbf{M} \eta} \mathbf{M}\mathbf{{C_\sim}}\mathbf{M} \xrightarrow{\epsilon\mathbf{M}} \mathbf{M}; \\
    & \mathbf{{C_\sim}} \xrightarrow{\eta \mathbf{C_\sim}} \mathbf{{C_\sim}}\mathbf{M}\mathbf{{C_\sim}} \xrightarrow{\mathbf{{C_\sim}}\epsilon} \mathbf{{C_\sim}}
\end{align*}
are identities. In other words, we want to show
\begin{align*}
    & \mathbf{M}(\mathcal{C}, \sim) \xrightarrow{\mathbf{M} (\eta_{(\mathcal{C}, \sim)})} \mathbf{M}(\mathcal{M}_{(\C, \sim)}, \sim_{dis})= \mathbf{M}(\C, \sim)\xrightarrow{\epsilon_{\mathbf{M}(\C, \sim)}} \mathbf{M}(\mathcal{C}, \sim) \\
    & \mathbf{{C_\sim}}(M) \xrightarrow{\eta_{ \mathbf{C_\sim}(M)}} \mathbf{{C_\sim}}(M) \xrightarrow{\mathbf{{C_\sim}}(\epsilon_M)} \mathbf{{C_\sim}}(M)
\end{align*}
are identities. Note that $\epsilon$ is identity, so we just need $\mathbf{M} (\eta_{(\mathcal{C}, \sim)})$ and $\eta_{ \mathbf{C_\sim}(M)}$ to be identities, which follows from the definitions of $\mathbf{M}, \mathcal{C}_{\sim}$ and $\eta$.
\end{proof}



\subsection{Sub-concentrations and quotient concentrations}\label{subsec: sub and quotient concentrations}

In this subsection we introduce sub-concentrations and quotient concentrations, which parallel the familiar concepts of submonoids and quotient monoids. We then prove the corresponding part of Theorem \ref{thm: concentration restrict to sub,quotient, and semidirect product}, showing that the concentration monoid functor $\mathbf{M}$ preserves these structures.


\begin{definition} \label{def: sub-concentration}
    Let $(\mathcal{C}, \sim)$ be a category with concentration, $\mathcal{B}$ be a sub-category of $\mathcal{C}$. We say $\sim$ is \textbf{closed} on $\mathcal{B}$ if $ \sim_{|_{\mathcal{B}}}$ (the equivalence relation restricted on $\textup{Mor}_{\mathcal{B}}$) is a concentration structure on $\mathcal{B}$.  $(\mathcal{B}, \sim_{|_{\mathcal{B}}})$ is called a \textbf{sub-concentration} of $(\mathcal{C}, \sim)$.
\end{definition}


Whenever we refer to a sub-concentration, we implicitly mean a sub-category with concentration. The following proposition directly follows from Definition \ref{def: sub-concentration}.

\begin{proposition} \label{prop: stabilizatin of sub concentration is submonoid}
    If $(\mathcal{B}, \sim _{|_{\mathcal{B}}})$ is a sub-concentration of $(\mathcal{C}, \sim)$, then $M_{(\mathcal{B}, \sim _{|_{\mathcal{B}}})}$ is a submonoid of $M_{(\mathcal{C}, \sim)}$.
\end{proposition}

One drawback of Definition \ref{def: sub-concentration} is that it is not clear when $\sim$ is closed on $\mathcal{B}$. The lemma below gives a sufficient condition.

\begin{lemma}
  Let $(\mathcal{C}, \sim)$ be a category with concentration, $\mathcal{B}$ be a sub-category of $\mathcal{C}$. If $f\in \textup{Mor}_{\mathcal{B}}$ and $f\sim g $ imply $ g\in \textup{Mor}_{\mathcal{B}}$, then $\sim$ is closed on $\mathcal{B}$.
\end{lemma}

\begin{proof}
    We just need to check that $\sim_{|_{\mathcal{B}}}$ satisfies the four axioms of concentration. It is easy to verify the first, the second and the fourth axiom, since $\sim_{|_{\mathcal{B}}}$ is a restriction of the given concentration structure $\sim$ and $\mathcal{B}$ is a subcategory of $\C$.

    Now we check the third axiom. Suppose that $f, g \in \textup{Mor}_{\mathcal{B}}$, then there exist $f' \sim f, g' \sim g$ such that $f' \circ g'$ exists. By the condition of the lemma, we know $f', g' \in \textup{Mor}_{\mathcal{B}}$, so $f' \sim_{|_{\mathcal{B}}} f$ and $g' \sim_{|_{\mathcal{B}}} g$.
\end{proof}


In practice, just as in other areas of mathematics, we sometimes want to identify a substructure with the image of a structure-preserving embedding. In our setting, this means viewing a category with concentration as sitting inside a larger one in a way that fully respects the concentration structure. The precise definition and explanation are given as follows.

\begin{definition}
    A concentration preserving functor $\iota: (\mathcal{B}, \sim_{\mathcal{B}}) \to (\mathcal{C}, \sim_{\mathcal{C}})$ is called an (concentration preserving) \textbf{embedding} if
    \begin{itemize}
        \item $\iota$ is injective on both objects and morphisms.
        \item $f \sim_{\mathcal{B}} g$ if and only if $\iota(f) \sim_{\mathcal{C}} \iota(g)$ for any $f, g \in \textup{Mor}_{\mathcal{B}}$.
    \end{itemize}
\end{definition}

Note that for any injective $\iota$ we obtain a concentration structure $\sim_{\iota(\mathcal{B})}$ on the image $\iota(\mathcal{B})$, where $\iota(f) \sim_{\iota(\mathcal{B})}\iota(g)$ if and only if $f\sim_{{\mathcal{B}}}g$. The following proposition follows immediately from the definitions.

\begin{proposition}
    Let $\iota: (\mathcal{B}, \sim_{\mathcal{B}}) \to (\mathcal{C}, \sim_{\mathcal{C}})$ be a concentration preserving embedding then  $(\iota({\mathcal{B}}),\sim_{\iota(\mathcal{B})})=(\iota({\mathcal{B}}),\sim_{\C |_{\iota({\mathcal{B})}}})$. In particular, $(\iota({\mathcal{B}}),\sim_{\iota(\mathcal{B})})$ is a sub-concentration of $(\C,\sim_{\C})$ and $M_{(\iota(\mathcal{B}), \sim_{\iota(\mathcal{B})})}$ is a submonoid of $M_{(\mathcal{C}, \sim_{\C})}$.
\end{proposition}

As a convention, we identify $(\mathcal{B}, \sim_{|_{\mathcal{B}}})$ and $(\iota(\mathcal{B}), \sim_{|_{\iota(\mathcal{B})}})$ when the embedding $\iota$ has no ambiguity in the context. With this convention, we say $(\mathcal{B}, \sim_{\mathcal{B}})$ is a sub-concentration of $(\mathcal{C}, \sim_{\mathcal{C}})$, and $M_{(\mathcal{B}, \sim_{\mathcal{B}})}$ is a submonoid of $M_{(\mathcal{C}, \sim_{\C})}$.

\vspace{5mm}

Before constructing quotient concentrations, we first need the notion of normal sub-concentrations. Just as normal subgroups in group theory, normal sub-concentrations serve as the appropriate setting for defining quotient concentrations.

\begin{definition}
    A sub-concentration $(\mathcal{B}, \sim_{|_{\mathcal{B}}})$ of $(\mathcal{C}, \sim)$ is called \textbf{normal}, if for any $f \in \textup{Mor}_{\mathcal{C}}, h \in \textup{Mor}_{\mathcal{B}}$, there exist some $h_1, h_2 \in \textup{Mor}_{\mathcal{B}}$ such that $[f][h] = [h_1][f]$ and $[h][f] = [f][h_2]$.
\end{definition}

Recall that a submonoid $S$ of a monoid $M$ is called normal if $xS = Sx$ for each $x \in M$. The next proposition shows that two notions of normality are equivalent under the concentration monoid functor.

\begin{proposition}
    $(\mathcal{B}, \sim_{|_{\mathcal{B}}})$ is a normal sub-concentration of $(\mathcal{C}, \sim)$ if and only if $M_{(\mathcal{B}, \sim_{|_{\mathcal{B}}})}$ is a normal submonoid of $M_{(\mathcal{C}, \sim)}$.
\end{proposition}

\begin{proof}
    For any $[f] \in M_{(\mathcal{C}, \sim_{|_{\mathcal{C}}})}$, note that $[f] M_{(\mathcal{B}, \sim_{|_{\mathcal{B}}})} \subseteq M_{(\mathcal{B}, \sim_{|_{\mathcal{B}}})}[f]$ if and only if for any $[h] \in M_{(\mathcal{B}, \sim_{|_{\mathcal{B}}})}$, there exists some $[h_1] \in M_{(\mathcal{B}, \sim_{|_{\mathcal{B}}})}$ such that $[f][h] = [h_1][f]$. Similarly, $M_{(\mathcal{B}, \sim_{|_{\mathcal{B}}})}[f] \subseteq [f] M_{(\mathcal{B}, \sim_{|_{\mathcal{B}}})}$ if and only if for any $[h] \in M_{(\mathcal{B}, \sim_{|_{\mathcal{B}}})}$, there exists some $[h_2] \in M_{(\mathcal{B}, \sim_{|_{\mathcal{B}}})}$ such that $[h][f] = [f][h_2]$. Then the proposition follows immediately.
\end{proof}

With the notion of normal sub-concentration, we can now define the corresponding quotient concentration.

\begin{definition}
Let $(\mathcal{B}, \sim_{|_{\mathcal{B}}})$ be a normal sub-concentration of $(\mathcal{C}, \sim)$. Define the \textbf{quotient concentration} $\sim_{/\mathcal{B}}$ on $\mathcal{C}$, such that $f \sim_{/\mathcal{B}} g$ if and only if there exist $h_1, h_2 \in \textup{Mor}_{\mathcal{B}}$ satisfying $[h_1][f] = [g][h_2]$.
\end{definition}

Before we show $\sim_{/\mathcal{B}}$ is a concentration, we first observe that if $f \sim g$, then for any object $A$ in $\mathcal{B}$, $[id_A][f] = [f] = [g] = [g][id_A]$, so $f \sim g$ implies that $f \sim_{/\mathcal{B}} g$.

\begin{lemma}\label{lem: quotien concentration is a concentration structure}
    $\sim_{/\mathcal{B}}$ is indeed a concentration structure on $\mathcal{C}$.
\end{lemma}

\begin{proof}
    We first check that $\sim_{/\mathcal{B}}$ is an equivalence relation. The reflexivity follows from that $[id_A][f] = [f][id_A]$ for any $f \in \textup{Mor}_{\mathcal{C}}$ and any $A \in \textup{Ob}_{\mathcal{B}}$.
    
    For the symmetry, suppose that $f \sim_{/\mathcal{B}} g$, then there exist $h_1, h_2 \in \textup{Mor}_{\mathcal{B}}$ satisfying $[h_1][f] = [g][h_2]$. Since $(\mathcal{B}x, \sim_{|_{\mathcal{B}}})$ is a normal sub-concentration, there exist $h_1', h_2' \in \textup{Mor}_{\mathcal{B}}$ such that $[h_1][f] = [f][h_1']$ and $[g][h_2] = [h_2'][g]$. Then $[h_2'][g] = [f][h_1']$, which means $g \sim_{/\mathcal{B}} f$.
    
    For the transitivity, suppose that $f \sim_{/\mathcal{B}} g$ and $g \sim_{/\mathcal{B}} h$, then there exist $h_1,h_2,h_3,h_4 \in \textup{Mor}_{\mathcal{B}}$ such that $[h_1][f] = [g][h_2]$ and $[h_3][g] = [h][h_4]$. Since $(\mathcal{B}, \sim_{|\mathcal{B}})$ is a sub-concentration, we can find $h_1', h_4'\in \textup{Mor}_{\mathcal{B}}$ such that $[h_1'] = [h_3][h_1]$ and $[h_4'] = [h_4][h_2]$. Then $[h_1'][f] = [h_3][h_1][f] = [h_3][g][h_2] = [h][h_4][h_2] = [h][h_4']$, and thus $f \sim_{/\mathcal{B}} h$.

    Next we check the four axioms of concentration.
    
    \textbf{Axiom 1.} For any $A, B \in \textup{Ob}_{\mathcal{C}}$, $id_A \sim_{/\mathcal{B}} id_B$, because $[h][id_A] = [h] = [h][id_B]$ for any $h \in \textup{Mor}_{\mathcal{B}}$.
    
    \textbf{Axiom 2.} Suppose that $f \sim_{/\mathcal{B}} f', g \sim_{/\mathcal{B}} g'$ and $f \circ g, f' \circ g'$ both exist, then there exist $h_1,h_2,h_3,h_4 \in \textup{Mor}_{\mathcal{B}}$ such that $[h_1][f] = [f'][h_2]$ and $[h_3][g] = [g'][h_4]$. Since $(\mathcal{B}, \sim_{|\mathcal{B}})$ is normal, there exist $h_5, h_6 \in \textup{Mor}_{\mathcal{B}}$ such that $[f][h_3] = [h_5][f]$, $[h_2][g'] = [g'][h_6]$. Then $[h_1][h_5][f \circ g] = [h_1][h_5][f][g] = [h_1][f][h_3][g] = [f'][h_2][g'][h_4] = [f'][g'][h_6][h_4] = [f' \circ g'][h_6][h_4]$. We can find $h_7, h_8 \in \textup{Mor}_{\mathcal{B}}$ such that $[h_7] = [h_1][h_5]$ and $[h_8] = [h_6][h_4]$, then $[h_7][f \circ g] = [f' \circ g'][h_8]$. Thus $f \circ g \sim_{/\mathcal{B}} f' \circ g'$.

    \textbf{Axiom 3.} It follows from the fact that $f \sim f'$ implies $f \sim_{/\mathcal{B}} f'$.

    \textbf{Axiom 4.} Given any morphisms $f\sim_{/\mathcal{B}} f', g\sim_{/\mathcal{B}} g', h\sim_{/\mathcal{B}} h', m\sim_{/\mathcal{B}} f\circ g, n \sim_{/\mathcal{B}} g'\circ h'$ with $f' \circ n, m \circ h$ exist,  we want to show $f' \circ n \sim_{/\mathcal{B}} m \circ h$. The definition of $\sim_{/\mathcal{B}}$ implies that there exist $h_1,h_1',h_2,h_2',h_3,h_3',h_4,h_4',h_5,h_5'$ in $\textup{Mor}_\mathcal{B}$ such that $[h_1][f]=[f'][h_1'],[h_2][g']=[g][h_2'],[h_3][h']=[h][h_3'],[h_4][m]=[f\circ g][h_4'], [h_5][n]=[g'\circ h'][h_5']$. By the normality of $(\mathcal{B}, \sim_{|\mathcal{B}})$, there exist $k_1,k_2,k_3,k_4$ in $\textup{Mor}_\mathcal{B}$ satisfying    
    \begin{align*}
        [k_1][f'\circ n]&=[f'][h_1'][k_2][h_2][k_3]([h_5][n])= [h_1][f][k_2][h_2]([k_3][g'])[h'][h_5']\\
        &=[h_1][f][k_2]([h_2][g'])([h_3][h'])[h_5']=[h_1][f]([k_2][g])[h_2'][h][h_3'][h_5']\\
        &=[h_1]([f][g][h_4'])[h_2'][h][h_3'][h_5']=[h_1][h_4][m][h_2'][h][h_3'][h_5']\\
        &=[m \circ h][k_4].
    \end{align*}
    
    Thus we have $f' \circ n \sim_{/\mathcal{B}} m \circ h$.

\end{proof}

Recall that given a monoid $M$, a congruence relation $\mathcal{R}$ on $M$ is an equivalence relation satisfying  that $a\mathcal{R}a'$ and $b\mathcal{R}b'$ implies $ab\mathcal{R}a'b'$ for any $a,a',b,b' \in M$. Given a congruence relation $\mathcal{R}$ on $M$ we can define a quotient monoid $S/\mathcal{R}$.

A normal submonoid $S$ of $M$ induces a congruence relation $\mathcal{R}_S$ on $M$ such that $a\mathcal{R}_Sb$ if and only if there exist $s_1,s_2\in S$ with $s_1a=bs_2$. It is not difficult to check that $\mathcal{R}_S$ is a congruence relation. In this case, we denote the quotient monoid $M/\mathcal{R}_S$ simply as $M/S$. 

\begin{remark}
    There is a more natural congruence relation $\mathcal{R}'_S$ induced by a normal submonoid $S$ of $M$, defined by $a\mathcal{R}'_Sb$ if and only if $Sa=bS$. It is easy to see that $a\mathcal{R}'_Sb$ implies $a\mathcal{R}_Sb$, but not vice versa in general. However, if $S$ is a subgroup of a group $M$, then $\mathcal{R}'_S=\mathcal{R}_S$, and $M/\mathcal{R}_S$ is the quotient group $M/S$ in the usual sense.   
\end{remark}

We now verify that taking concentration monoids is compatible with the quotient structures.

\begin{proposition} \label{prop: concentration monoid of quotient concentration is the quotient of monoid}
    Let $(\mathcal{B}, \sim_{|_{\mathcal{B}}})$ be a normal sub-concentration of $(\mathcal{C}, \sim)$, then $M_{(\mathcal{C}, \sim_{/\mathcal{B}})} \cong M_{(\mathcal{C}, \sim)} / M_{(\mathcal{B}, \sim_{|_{\mathcal{B}}})}$.
\end{proposition}

\begin{proof}
    To distinguish different equivalence relations, we will denote the elements in $M_{(\mathcal{C}, \sim)}$ as $[f]$, the elements in $M_{(\mathcal{C}, \sim)} / M_{(\mathcal{B}, \sim_{|_{\mathcal{B}}})}$ as $[[f]]$, and the elements in $M_{(\mathcal{C}, \sim_{/\mathcal{B}})}$ as $[f]_{/\mathcal{B}}$.
    
    Consider the map $\phi: M_{(\mathcal{C}, \sim_{/\mathcal{B}})} \to M_{(\mathcal{C}, \sim)} / M_{(\mathcal{B}, \sim_{|_{\mathcal{B}}})}$ sending $[f]_{/\mathcal{B}}$ to $[[f]]$, we claim that $\phi$ is a monoid isomorphism. First, we check $\phi$ is well defined.  Suppose that $[f]_{/\mathcal{B}} = [g]_{/\mathcal{B}}$, then $[h_1][f] = [g][h_2]$ for some $h_1, h_2 \in \textup{Mor}_{\mathcal{B}}$. Note that since $[h_1], [h_2] \in M_{(\mathcal{B}, \sim_{|_{\mathcal{B}}})}$, we have $[[f]] = [[g]]$.

    Next we check that $\phi$ is a monoid homomorphism. Let $[f]_{/\mathcal{B}}, [g]_{/\mathcal{B}}$ be two elements in $M_{(\mathcal{C}, \sim_{/\mathcal{B}})}$. Since $[f]_{/\mathcal{B}}[g]_{/\mathcal{B}} = [f' \circ g']_{/\mathcal{B}}$ for some $f' \sim_{/\mathcal{B}} f$ and $g' \sim_{/\mathcal{B}} g$, we have $\phi([f]_{/\mathcal{B}}[g]_{/\mathcal{B}})=[[f' \circ g']]$. On the other hand, $[f][g] = [f'' \circ g'']$ for some $f'' \sim f$ and $g'' \sim g$, so $\phi([f]_{/\mathcal{B}}) \phi([g]_{/\mathcal{B}}) = [[f]][[g]] = [[f][g]] = [[f'' \circ g'']]$. Note that by the observation before Lemma \ref{lem: quotien concentration is a concentration structure} and the transitivity of  $\sim_{/\mathcal{B}}$, we have $f' \sim_{/\mathcal{B}}f''$ and $g' \sim_{/\mathcal{B}} g''$, implying that $f' \circ g' \sim_{/\mathcal{B}} f'' \circ g''$. By the well-definedness of $\phi$, we have $[[f' \circ g']] = [[f'' \circ g'']]$, so $\phi([f]_{/\mathcal{B}}[g]_{/\mathcal{B}})=\phi([f]_{/\mathcal{B}}) \phi([g]_{/\mathcal{B}}).$

    Lastly, we check that $\phi$ is a bijection. The surjectivity is obvious. For the injectivity, suppose that $[[f]] = [[g]]$, then $[h_1][f] = [g][h_2]$ for some $[h_1], [h_2] \in M_{(\mathcal{B}, \sim_{|_{\mathcal{B}}})}$. Then $[f]_{/\mathcal{B}} = [g]_{/\mathcal{B}}$ follows from the definition of $\sim_{/\mathcal{B}}$.
\end{proof}

We end this subsection with an example of sub-concentration and quotient concentration.

\begin{example} \label{ex: Z/2 and Z/4 category with concentration.2}

Consider the category with concentration $(\C, \sim_a)$ in Example \ref{ex: Z/2 and Z/4 category with concentration}. For convenience, we simply denote $\sim_a$ as $\sim$.

We consider the subcategory $\mathcal{B}$ of $\C$, with a single object $D$, and morphisms $0_D$ and $2_D$. Then $(\mathcal{B},\sim_{dis})$ is a sub-concentration of $(\C,\sim_a)$, and $M_{(\mathcal{B},\sim_{dis})}\cong \mathbb{Z}/2$ is a subgroup of $M_{(\C,\sim)}\cong \mathbb{Z}/4$.

It is not hard to verify that $(\mathcal{B},\sim_{dis})$ is normal. The corresponding quotient concentration $\sim_{/\mathcal{B}}$ is given by $0_C \sim_{/\mathcal{B}}  0_D \sim_{/\mathcal{B}} 2_D \sim_{/\mathcal{B}} 1_C$. The quotient concentration group is $M_{(\C,\sim_{/\mathcal{B}})}\cong M_{(\C,\sim)}/M_{(\mathcal{B},\sim_{dis})} \cong \mathbb{Z}/2$. 

\end{example}





\subsection{More examples of concentration structures}\label{subsec: mroe examples}

At the end of Section \ref{subsec:n-concentration} we gave some basic examples of concentration structures. In this subsection we provide more examples to further illustrate the idea of concentration structures. 

\begin{example} \label{ex: Z2 Z4 with one morphism}
    Consider a category with two objects, $C$ and $D$, where $\textup{Mor}(C, C) = \mathbb{Z}/2$ and $\textup{Mor}(D, D) = \mathbb{Z}/4$. Moreover, there is a unique morphism from $C$ to $D$ (see Figure \ref{fig: Z/2 and Z/4 example2}), which differs from the category in Example \ref{ex: Z/2 and Z/4 category with concentration}. In this case, it turns out that the trivial concentration structure is the only possible concentration structure.    
\end{example}
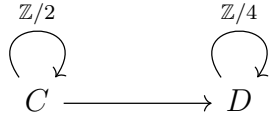
\begin{figure}[h]
    \[\begin{tikzcd}
	C && D
	\arrow["{\mathbb{Z}/2}", from=1-1, to=1-1, loop, in=55, out=125, distance=10mm]
	\arrow[from=1-1, to=1-3]
	\arrow["{\mathbb{Z}/4}", from=1-3, to=1-3, loop, in=55, out=125, distance=10mm]
\end{tikzcd}\]
\caption{The category in Example \ref{ex: Z2 Z4 with one morphism}.}
\label{fig: Z/2 and Z/4 example2}
\end{figure}

\begin{example} \label{ex: z2 times z2}
    In this example, we consider the category shown in Figure \ref{fig:z/2 times z/2}. For $i,j \in \{C,D\}$, The morphisms from $i$ to $j$ are of the form $x_i^j$, where $x=r,b,g,d$ denotes red, blue, green and black, respectively. Moreover, the compositions are defined as follows, whenever the composition exists.
    \begin{align*}
        &r_j^k \circ r_i^j = b_j^k \circ b_i^j= g_j^k \circ g_i^j=d_j^k\circ d_i^j=d_i^k \\
        &r_j^k \circ b_i^j = b_j^k \circ r_i^j=g_j^k \circ d_i^j=d_j^k \circ g_i^j=g_i^k\\
        &b_j^k \circ g_i^j = g_j^k \circ b_i^j=r_j^k \circ d_i^j=d_j^k \circ r_i^j=r_i^k\\
        & g_j^k \circ r_i^j = r_j^k \circ g_i^j=b_j^k \circ d_i^j=d_j^k \circ b_i^j=b_i^k
    \end{align*}
    
    Next, we consider the equivalence relation obtained by identifying morphisms of the same color. It is straightforward to check that this equivalence is a $3$-concentration, and the corresponding concentration monoid is $\mathbb{Z}/2 \times \mathbb{Z}/2$.   
\end{example}

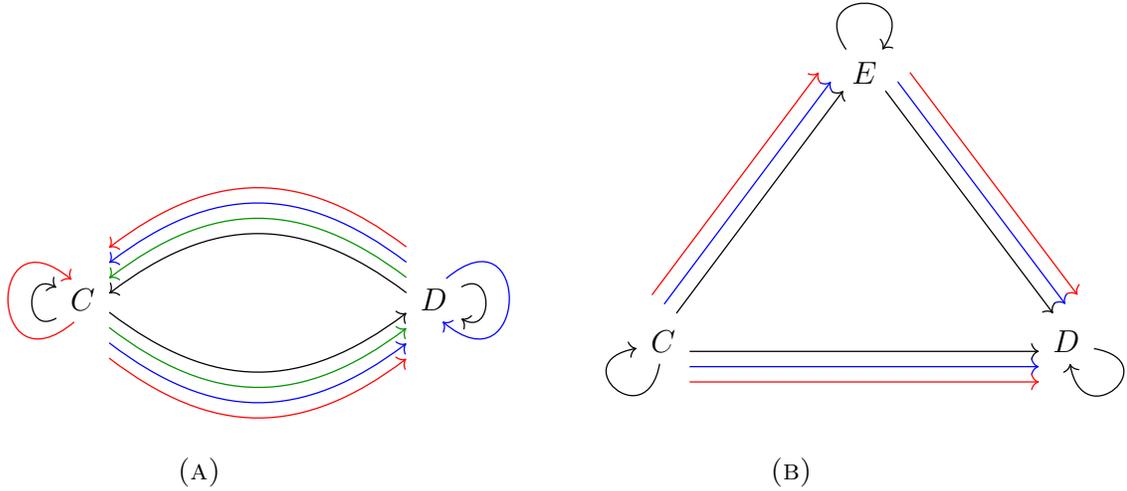
\begin{figure}
    \begin{subfigure}{0.4\textwidth}
        \centering
        \begin{tikzcd}
    	C &&&& D
    	\arrow[shift right=2, color=red, from=1-1, to=1-1, loop, in=140, out=220, distance=15mm]
    	\arrow[from=1-1, to=1-1, loop, in=150, out=210, distance=5mm]
    	\arrow[shift right=3, color=blue, curve={height=30pt}, from=1-1, to=1-5]
    	\arrow[shift right, color={rgb,255:red,0;green,148;blue,0}, curve={height=30pt}, from=1-1, to=1-5]
    	\arrow[shift right=5, color=red, curve={height=30pt}, from=1-1, to=1-5]
    	\arrow[shift left, curve={height=30pt}, from=1-1, to=1-5]
    	\arrow[shift right, color={rgb,255:red,0;green,148;blue,0}, curve={height=30pt}, from=1-5, to=1-1]
    	\arrow[shift left, curve={height=30pt}, from=1-5, to=1-1]
    	\arrow[shift right=5, color=red, curve={height=30pt}, from=1-5, to=1-1]
    	\arrow[shift right=3, color=blue, curve={height=30pt}, from=1-5, to=1-1]
    	\arrow[from=1-5, to=1-5, loop, in=330, out=30, distance=5mm]
    	\arrow[shift right=2, color=blue, from=1-5, to=1-5, loop, in=320, out=40, distance=15mm]
        \end{tikzcd}
        \caption{}
        \label{fig:z/2 times z/2}
    \end{subfigure}
    \hspace{1.5cm}
    \begin{subfigure}{0.4\textwidth}
        \centering
        \begin{tikzcd}
    	&& E \\
    	\\
    	\\
    	\\
    	C &&&& D
    	\arrow[from=1-3, to=1-3, loop, in=55, out=125, distance=10mm]
    	\arrow[shift left=3, draw=blue, from=1-3, to=5-5]
    	\arrow[shift left=5, draw=red, from=1-3, to=5-5]
    	\arrow[shift left, from=1-3, to=5-5]
    	\arrow[shift left=3, draw=blue, from=5-1, to=1-3]
    	\arrow[shift left=5, draw=red, from=5-1, to=1-3]
    	\arrow[shift left, from=5-1, to=1-3]
    	\arrow[from=5-1, to=5-1, loop, in=190, out=260, distance=10mm]
    	\arrow[shift right=3, draw=blue, from=5-1, to=5-5]
    	\arrow[shift right, from=5-1, to=5-5]
    	\arrow[shift right=5, draw=red, from=5-1, to=5-5]
    	\arrow[from=5-5, to=5-5, loop, in=280, out=350, distance=10mm]
        \end{tikzcd}
        \caption{}
        \label{fig: z/3}
    \end{subfigure}
    \caption{(A) The category in Example \ref{ex: z2 times z2}. (B) The category in Example \ref{ex: not 3-concentration}.}
\end{figure}

\begin{example} \label{ex: not 3-concentration}
    All the concentration structures in the previous examples are $3$-concentration structures. Next, we give an example of a concentration structure that is not a $3$-concentration structure. Consider the category shown in Figure \ref{fig: z/3}, which contains three objects $C, D$ and $E$. For $i,j \in \{C,D,E\}$, the morphism from $i$ to $j$ are of the form $x_i^j$, where $x=r,b,d$ denotes red, blue and black, respectively. The compositions are defined as follows, whenever the composition exists.
    \begin{align*}
        &d_j^k\circ d_i^j = b_j^k \circ r_i^j = r_j^k \circ b_i^j=d_i^k \\
        &b_j^k \circ b_i^j = r_j^k \circ d_i^j = d_j^k \circ r_i^j=r_i^k\\
        &r_j^k \circ r_i^j = b_j^k \circ d_i^j = d_j^k \circ b_i^j=b_i^k
    \end{align*}
    
    We again consider the equivalence relation obtained by identifying morphisms of the same color. Notice that the three red arrows are not $3$-composable, so this relation is not a $3$-concentration. However, it is a concentration, and its concentration monoid is $\mathbb{Z}/3$.
\end{example}

\section{Concentrations and semidirect products} \label{sec: Concentrations and semidirect products}

In this section we define semidirect products of categories with concentration, and study their relation to semidirect products of monoids.  

We start with some notations we will use later. Given a category $\mathcal{C}$, let $\textup{End}(\mathcal{C})$ denote the monoid of functors from $\mathcal{C}$ to itself, and let $\mathcal{E}nd(\C)$ denote the category witha a single object $*$ and $\textup{Mor}(*,*) = \textup{End}(\mathcal{C})$. Similarly, let $\textup{Aut}(\mathcal{C})$ denote the group of strongly invertible functors from $\mathcal{C}$ to itself, and let $\mathcal{A}ut(\C)$ denote the category with a single object $*$ and $\textup{Mor}(*,*) = \textup{Aut}(\mathcal{C})$.

\subsection{Semidirect products of categories} \label{sec: Semidirect products of categories}

The semidirect product of categories was defined in \cite{MR1715405}, using partial full endofunctors. Our definition is slightly different: we use endofunctors and follow a more direct analogue of the classical semidirect product in group theory. To the best of our knowledge, this exact formulation has not previously appeared in the literature.

\begin{definition} \label{def: semi product of two categories}
    Let $\C$, $\D$ be two categories, and let $\Phi:\D \rightarrow\mathcal{E}nd(\C)$ be a functor. The \textbf{semidirect product} of $\C$ and $\D$, denoted $\C \rtimes_{\Phi}\D$, is a category defined as follows:
    \begin{itemize}
        \item $\textup{Ob}( \C \rtimes_{\Phi}\D)=\textup{Ob}(\C)\times \textup{Ob}(\D)$.
        \item $\textup{Mor}((C_1,D_1),(C_2,D_2))=\{(\alpha,f): f\in \textup{Mor}(D_1,D_2),\alpha\in \textup{Mor}(\Phi_f(C_1),C_2))$, where $\Phi_f(C_1)=\Phi(f)(C_1)$.
        \item Composition is given by  $(\alpha_2,f_2)\circ(\alpha_1,f_1)=(\alpha_2\circ\Phi_{f_2}(\alpha_1),f_2\circ f_1)$.
    \end{itemize}  
\end{definition}

We provide a brief justification to confirm that $\C \rtimes_{\Phi}\D$ is a well-defined category.

\begin{lemma}
   $\C \rtimes_{\Phi}\D$ is a well-defined category.
\end{lemma}

\begin{proof}
    We first check that the composition is well defined. Suppose that $(\alpha_1, f_1) \in \textup{Mor}((C_1, D_1), (C_2, D_2))$ and $(\alpha_2, f_2) \in \textup{Mor}((C_2, D_2), (C_3, D_3))$, we want to show $(\alpha_2\circ\Phi_{f_2}(\alpha_1),f_2\circ f_1)\in \textup{Mor}((C_1,D_1),(C_3,D_3))$, that is, $\alpha_2\circ\Phi_{f_2}(\alpha_1) \in \text{Mor}(\Phi_{f_2 \circ f_1}(C_1), C_3)$ and $f_2 \circ f_1 \in \text{Mor}(D_1,D_3)$.
    
    By definition, $\alpha_{1}\in \text{Mor}(\Phi_{f_1}(C_1),C_2)$ and $\alpha_2 \in \textup{Mor}(\Phi_{f_2}(C_2), C_3)$. Since $\Phi_{f_2}(\alpha_1) \in \textup{Mor}(\Phi_{f_2}(\Phi_{f_1}(C_1)), \Phi_{f_2}(C_2))$, it follows that $\alpha_2\circ\Phi_{f_2}(\alpha_1)$ is a morphism from $\Phi_{f_2 \circ f_1}(C_1)$ to $C_3$. Moreover, by definition $f_1 \in \textup{Mor}(D_1, D_2)$ and $f_2 \in \textup{Mor}(D_2, D_3)$, so $f_2 \circ f_1$ is a morphism from $D_1$ to $D_3$.
    
    Associativity of the composition follows from a similar argument as the associativity of multiplication in the semidirect product of groups.
    
    It is easy to check that the identity morphism $id_{(C, D)}$ is given by $(id_C, id_D)$.
\end{proof}

\subsection{Semidirect products with concentration structures}

In this subsection, we extend the semidirect product construction to categories equipped with concentration structures.

\begin{definition}
    Let $(\C,\sim_\C)$,and $(\D,\sim_\D)$ be two categories with concentrations, and let $\Phi:\D \rightarrow\mathcal{A}ut(\C)$ be a functor \textbf{compatible} with the two concentrations, meaning that whenever $\alpha \sim_\C \alpha'$ and $f\sim_\D f'$, we have $\Phi_f(\alpha)\sim_\C \Phi_{f'}(\alpha')$. The \textbf{semidirect product} of $(\C,\sim_\C)$ and $(\D,\sim_\D)$ is defined to be the category $\C \rtimes_{\Phi}\D$ in Definition \ref{def: semi product of two categories}, with the concentration $\sim_{ \rtimes_{\Phi}}$ given by $(\alpha, f) \sim_{ \rtimes_{\Phi}} (\alpha', f')$ if and only if $\alpha \sim_{\mathcal{C}} \alpha'$ and $f \sim_{\mathcal{D}} f'$.
\end{definition}

\begin{lemma}
    $\sim_{ \rtimes_{\Phi}}$ is indeed a concentration structure on $\C \rtimes_{\Phi}\D$.
\end{lemma}

\begin{proof}
    It is obvious that $\sim_{ \rtimes_{\Phi}}$ is an equivalence relation on $\text{Mor}(\C \rtimes_{\Phi}\D)$, since both $\sim_\C$ and $\sim_\D$ are equivalence relations. We now check the four axioms of concentration. 
    
    \textbf{Axiom 1.} Let $(id_{C_1},id_{D_1})$ and $(id_{C_2},id_{D_2})$ be identity morphisms. Then $id_{C_1} \sim_\C id_{C_2}$ and $id_{D_1} \sim_\D id_{D_2}$ implies $(id_{C_1},id_{D_1}) \sim_{\rtimes_\Phi} (id_{C_2},id_{D_2})$.

    \textbf{Axiom 2.} Suppose $(\alpha,f)\sim_{\rtimes\Phi}(\alpha',f')$ and $(\beta,g)\sim_{\rtimes\Phi}(\beta',g')$ and both $(\alpha,f)\circ(\beta,g)$, $(\alpha',f')\circ(\beta',g')$ exist. Then \begin{align*}
        (\alpha,f)\circ(\beta,g)&=(\alpha\circ \Phi_g(\beta),f\circ g) \\
        &\sim_{\rtimes\Phi} (\alpha'\circ \Phi_{g'}(\beta'),f'\circ g' ) \\
        &=(\alpha',f')\circ(\beta',g')
    \end{align*}
    where $\Phi_g(\beta) \sim_\C \Phi_{g'}(\beta')$ is from the compatibility of $\Phi$.

    \textbf{Axiom 3.} Let $(\alpha,f),(\beta,g) \in \text{Mor}(\C \rtimes_{\Phi}\D)$. We can first find $f', g'\in \text{Mor}(\D)$ such that $f\sim_\D f'$, $g \sim_\D g'$ and $f'\circ g'$ exists. Then we find $\alpha'', \beta' \in \text{Mor}(\C)$ such that $\alpha'' \sim_\C\Phi_{f'}^{-1}(\alpha)$, $\beta' \sim_C \beta$ and $\alpha'' \circ \beta'$ exists. Let $\alpha' = \Phi_{f'}(\alpha'')$.

    Next we show that $(\alpha',f') \sim_{\rtimes\Phi} (\alpha,f)$, $(\beta',g')\sim_{\rtimes\Phi} (\beta,g)$ such that $(\alpha',f') \circ (\beta',g')$ exists. Note that $(\beta',g')\sim_{\rtimes\Phi} (\beta,g)$ directly follows from the definition of $\sim_{\rtimes\Phi}$. To show that $(\alpha',f') \sim_{\rtimes\Phi} (\alpha,f)$, we just need to check $\alpha' \sim_\C \alpha$. Indeed, since $\Phi$ is compatible with the concentration structures, we have $\alpha' = \Phi_{f'}(\alpha'') \sim_\C \Phi_{f'}(\Phi_{f'}^{-1}(\alpha)) = \alpha$.

    Now the existence of the composition follows from the following.
    \begin{align*}
        (\alpha',f') \circ (\beta',g')&= (\alpha' \circ \Phi_{f'} (\beta'), f'\circ g')\\
        &= (\Phi_{f'}(\alpha'')\circ \Phi_{f'}(\beta'), f'\circ g') \\
        &= (\Phi_{f'}(\alpha'' \circ \beta'), f'\circ g'),
    \end{align*}

    \textbf{Axiom 4.} The associativity axiom directly follows from the definition of $\sim_{\rtimes\Phi}$.

\end{proof}

\subsection{Concentration monoids of semidirect products}
\label{sec: stabiliztion monoids and semidirect products of monoids}

Let $(\C,\sim_\C)$and $(\D,\sim_\D)$ be two categories with concentrations, let $M_{(\C,\sim_\C)}$ and $M_{(\D,\sim_\D)}$ be their concentration monoids. Given a functor $\Phi: \D \rightarrow\mathcal{A}ut(\C)$ compatible with the concentration structures of $\C$ and $\D$, the functor $\Phi$ induces a well-defined homomorphism $\phi: M_{(\D, \sim_{\D})} \to \textup{Aut}(M_{(\C, \sim_\C)})$ such that $(\phi[f])[\alpha] = [\Phi_f(\alpha)]$. Its well-definedness comes from the compatibility condition of $\Phi$.

Recall that given two monoids $M$ and $N$ and a homomorphism $\phi: N \to \textup{Aut}(M)$, the semidirect product $M \rtimes_{\phi} N$ is defined to be $M \times N$ as a set, with multiplication $(m_1, n_1) (m_2, n_2) = (m_1 \phi_{n_1}(m_2), n_1 n_2)$. 


We are now ready to prove the final part of Theorem \ref{thm: concentration restrict to sub,quotient, and semidirect product}, showing that the concentration monoid functor $\mathbf{M}$ preserves semidirect products.

\begin{proposition} \label{prop: concentration of semidirect product is semidirect product of concentration}
    $M_{(\C \rtimes_{\Phi}\D, \sim_{\rtimes_{\Phi}})} \cong M_{(\C, \sim_\C)} \rtimes_{\phi} M_{(\D, \sim_\D)}$
\end{proposition}

\begin{proof}
    Consider the map $\theta: M_{(\C \rtimes_{\Phi}\D, \sim_{\rtimes_{\Phi}})} \to M_{(\C, \sim_\C)} \rtimes_{\phi} M_{(\D, \sim_\D)}$ sending $[\alpha,f]$ to $([\alpha], [f])$. We claim that $\theta$ is a well-defined isomorphism.

    First, $\theta$ is a well-defined bijection because $(\alpha, f) \sim_{\rtimes_{\Phi}} (\alpha', f')$ if and only if $\alpha \sim_\C \alpha'$ and $f \sim_\D f'$. Then we just need to check that $\theta$ is a homomorphism. Given $[\alpha, f], [\beta, g] \in M_{\C \rtimes_{\Phi}\D}$, suppose that $[\alpha, f][\beta, g] = [(\alpha', f') \circ (\beta', g')]$ where $\alpha \sim_\C \alpha'$, $\beta \sim_\C \beta'$ and $f \sim_\D f'$, $g \sim_\D g'$, then
    \begin{align*}
        \theta([\alpha, f][\beta, g])
        =& \theta[(\alpha', f') \circ (\beta', g')] \\
        =& \theta[\alpha' \circ \Phi_{f'}(\beta'), f' \circ g'] \\
        =& ([\alpha' \circ \Phi_{f'}(\beta')],[f' \circ g']) \\
        =& ([\alpha][\Phi_{f}(\beta)],[f][g]) \\
        =& ([\alpha](\phi[f][\beta]),[f][g]) \\
        =& ([\alpha], [f]) ([\beta], [g]) \\
        =& \theta[\alpha, f] \theta[\beta, g]
    \end{align*}
\end{proof}

\section{$G$-equivariant direct limits}\label{sec: Concentrations and $G$-equivariant direct limits}

In this section, we reinterpret direct limits of groups using concentration structures, then extend it to define $G$-equivariant direct limits for group $G$. Using the semidirect product construction in Section \ref{sec: Concentrations and semidirect products}, we show that such a $G$-equivariant direct limit decomposes naturally as a semidirect product of the original direct limit and the group $G$. As an application, we give an explicit construction of $\mathbb{R}$-braid groups as $\textup{Aut}^+(\mathbb{R})$-equivariant direct limits, where $\textup{Aut}^+(\mathbb{R})$ denotes the group of order preserving automorphisms of $\mathbb{R}$, viewed either as a set, a topological space or a smooth manifold.


\vspace{5mm}

We begin with a brief review of direct limits of groups. A directed set $(S, \leq)$ is a partially ordered set, in which for any $A,B\in S$ there exists $C\in S$ such that $A\leq C$ and $B \leq C$.

Then we consider the associated \textbf{direct category} $\mathcal{S}$, defined by $\textup{Ob}(\mathcal{S}) = S$ with a single morphism $\iota_A^B$ from $A$ to $B$ whenever $A \leq B$. The composition is given by $\iota_B^C \circ \iota_A^B = \iota_A^C$. 

Given a functor $F:\mathcal{S}\rightarrow \mathscr{G}rp$ from $\mathcal{S}$ to the category of groups, we naturally obtain a direct system $(\{F(A)\}_{A\in S}, \{F(\iota_A^B)\}_{A,B\in S})$ with directed set $(S,\leq)$. It is well known that the direct limit $\displaystyle\lim_{\longrightarrow} \{F(A)\}_{A \in S}$ coincides with the colimit $\displaystyle\lim_{\longrightarrow} F$.

\subsection{$G$-equivariant direct limit as a concentration monoid}

Recall that a group $G$ acting on a category $\C$ is specified as a group homeomorphism $\rho: G \rightarrow \textup{Aut}(\mathcal{C})$, where $\textup{Aut}(\mathcal{C})$ denotes the group of strongly invertible functors from $\mathcal{C}$ to itself. When the action has no ambiguity, we simply write $ \rho(f)(\_)$ as $f(\_)$ for the action of $f\in G$ on an object or a morphism of $\C$. A functor $F: \mathcal{C} \to \mathcal{D}$ is said to be \textbf{$G$-equivariant} if $F(f(C)) = f(F(C))$ for any object $C$ in $\mathcal{C}$, and $F(f(\alpha)) = g(F(\alpha))$ for any morphism $\alpha$ in $\mathcal{C}$.

Let $(S, \leq)$ be a directed set, and let $\mathcal{S}$ be its associated direct category. Suppose a group $G$ acts on $\mathcal{S}$. Note that the morphisms in $\mathcal{S}$ are determined by the partial order on the objects, so the action must satisfy $f(\iota_A^B)=\iota_{f(A)}^{f(B)}$ for any $f\in G$ and any $A \leq B$. 


In the next two definitions, we construct a category with concentration that provides a reinterpretation of the direct limit in the presence of a $G$-action.

\begin{definition} \label{def: category S_G}
    Given a group action of $G$ on $\mathcal{S}$, and a $G$-equivariant functor $F: \mathcal{S} \to \mathscr{G}rp$ (where $G$ acts trivially on $\mathscr{G}rp$), define the category $\mathcal{S}_{G}$ as follows.
    \begin{itemize}
        \item $\textup{Ob}(\mathcal{S}_G) = \textup{Ob}(\mathcal{S})$.
        \item $\textup{Mor}(B, A) = \{(A, \alpha, f): \alpha \in F(A), f \in G, f(B) = A\}$.
        \item Composition is given by $(A, \alpha, f) \circ (B, \beta, g)  = (A, \alpha\beta, fg)$. Here $\alpha\beta$ is well defined because $F$ acts $G$-equivariantly on the objects, i.e. $F(A) = F(f(B)) = F(B)$. 
    \end{itemize}

    The identity morphism $id_{A}$ is given by $(A, e_{F(A)}, \epsilon)$, where $e_{F(A)}$ is the identity in $F(A)$ and $\epsilon$ is the identity in $G$. 
\end{definition}

\begin{remark}
    Each morphism $(A, \alpha, f)$ in $\mathcal{S}_G$ implicitly determines its source and target: the target is $A$, while the source is $f^{-1}(A)$. The element $\alpha$ belongs to $F(A)$, the group associated with the target object $A$.
\end{remark}

\begin{definition} \label{def: concentration on S_G}
    We define the concentration structure on $\mathcal{S}_G$ such that $(A, \alpha, f) \sim (B, \beta, g)$ if and only if
    \begin{enumerate}
        \item $(F(\iota_A^C))(\alpha) = (F(\iota_B^C))(\beta)$ for some $C \in \textup{Ob}(\mathcal{S})$.
        \item $f = g$.
    \end{enumerate}
\end{definition}

\begin{lemma}
    $\sim$ is indeed a concentration structure on $\mathcal{S}_G$.
\end{lemma}

\begin{proof}
    We first verify that $\sim$ is an equivalence relation. Its reflexivity and symmetry are obvious. For the transitivity, suppose that $(A, \alpha, f) \sim (B, \beta, f)$ and $(B, \beta, f) \sim (C, \gamma, f)$, then there exist $B',B'' \in \textup{Ob}(\mathcal{S})$ such that $(F(\iota_A^{B'}))(\alpha) = (F(\iota_B^{B'}))(\beta)$ and $(F(\iota_C^{B''}))(\gamma) = (F(\iota_B^{B''}))(\beta)$. Since $\textup{Ob}(\mathcal{S})$ is a directed set, we can choose an object $B'''$ with $B' \leq B'''$ and $B'' \leq B'''$. Then
    \begin{align*}
    (F(\iota_A^{B'''}))(\alpha)&= (F(\iota_{B'}^{B'''} \circ \iota_A^{B'}))(\alpha) \\
    & = (F(\iota_{B'}^{B'''} \circ\iota_B^{B'}))(\beta) \\
    &= (F(\iota_{B}^{B'''}))(\beta) \\
    &= (F(\iota_{B''}^{B'''} \circ\iota_B^{B''}))(\beta) \\
    &
    =(F(\iota_{B''}^{B'''} \circ \iota_C^{B''}))(\gamma) \\
    &=(F(\iota_C^{B'''}))(\gamma)
    \end{align*}


    Next we check $\sim$ satisfies the three axioms of $3$-concentration structure (Definition \ref{def: n-concentration}). Then it follows from Proposition \ref{prop: 3-concentration implies concentration} that $\sim$ is a concentration structure.

    \textbf{Axiom 1.} For any two objects $A, B$, choose $C \geq A, B$. Then $(F(\iota_A^C))(e_{F(A)})=e_{F(C)}=(F(\iota_B^C))(e_{F(B)})$, so $(A, e_{F(A)}, \epsilon) \sim (B, e_{F(B)}, \epsilon)$
    
    \textbf{Axiom 2.} Suppose $(A,\alpha,f)\sim(A',\alpha',f)$ and $(B,\beta,g)\sim(B',\beta',g)$ and assume both $(A,\alpha,f)\circ(B,\beta,g)$, $(A',\alpha',f')\circ(B',\beta',g')$ exist. Then there exist objects $C, D$ with $(F(\iota_A^C))(\alpha)=(F(\iota_{A'}^C))(\alpha')$ and $(F(\iota_B^D))(\beta)=(F(\iota_{B'}^D))(\beta')$. First choose $\overline{E}$ such that $\overline{E}\geq C$ and $\overline{E}\geq D$, then choose $E$ such that $E\geq f(\overline{E})$ and $E\geq \overline{E}$. Then we have
    \begin{align*}
        (F(\iota_A^E))(\alpha\beta)
        &=(F(\iota_A^E))(\alpha)(F(\iota_A^E))(\beta) \\ 
        &=(F(\iota_A^E))(\alpha)(F(\iota_{B}^{f^{-1}(E)}))(\beta) \\
        &= (F(\iota_C^E\circ \iota_A^C))(\alpha)(F(\iota_D^{f^{-1}(E)}\circ \iota_B^D))(\beta)\\
        &=(F(\iota_C^E\circ \iota_{A'}^C))(\alpha')(F(\iota_D^{f^{-1}(E)}\circ \iota_{B'}^D))(\beta')\\
        &=(F(\iota_{A'}^E))(\alpha')(F(\iota_{B'}^{f^{-1}(E)}))(\beta')\\
        &=(F(\iota_{A'}^E))(\alpha')(F(\iota_{A'}^E))(\beta') \\ 
        &= (F(\iota_{A'}^E))(\alpha'\beta')
    \end{align*}
    where $F(\iota_A^E) = F(\iota_{B}^{f^{-1}(E)})$ and $F(\iota_{A'}^E) = F(\iota_{B'}^{f^{-1}(E)})$ because $F$ acts $G$-equivariantly on morphisms and $f(B) = A, f(B') = A$. It follows that $(A,\alpha\beta,fg)\sim (A',\alpha'\beta',fg)$.

    \textbf{Axiom 3.} For any three morphisms $(A, \alpha,f),(B, \beta,g), (C,\gamma,h)$, choose an object $D$ such that $A \leq D$ and $f(B) \leq D$ and $f(g(C)) \leq D$. Then $(A, \alpha,f) \sim (D, F(\iota_A^D)(\alpha), f)$, which is a morphism from $f^{-1}(D)$ to $D$. Similarly, $(B, \beta, g)$ is equivalent to a morphism from $(fg)^{-1}(D)$ to $f^{-1}(D)$, and $(C, \gamma, h)$ is equivalent to a morphism from $(fgh)^{-1}(D)$ to $(fg)^{-1}(D)$. It ensures that the third axiom is satisfied.
    
\end{proof}

With the category $\mathcal{S}_G$ and its concentration structure, we can now define the $G$-equivariant direct limit.

\begin{definition} \label{def: G-equivariant direct limit}
    Given a direct system $(\{F(A)\}_{A \in S}, \{F(\iota_A^B)\}_{A, B \in S})$, a $G$-action on the direct category $\mathcal{S}$, and a $G$-equivariant functor $F: \mathcal{S} \to \mathscr{G}rp$, we define the \textbf{$G$-equivariant direct limit} $\displaystyle \lim_{\longrightarrow}{}^G F$ to be $ M_{(\mathcal{S}_G,\sim)}$, the concentration group of $(\mathcal{S}_G, \sim)$.
\end{definition}

The following proposition shows that our definition of the $G$-equivariant direct limit extends the classical direct limit, recovering it when $G$ is trivial.

\begin{proposition} \label{prop: trivial G-equivariant semidirect product is direct product}
    If $G$ acts trivially on $\mathcal{S}$, then $\displaystyle \lim_{\longrightarrow}{}^G F \cong \displaystyle \lim_{\longrightarrow} F \times G$. In particular, $\displaystyle \lim_{\longrightarrow}{}^0 F \cong \displaystyle \lim_{\longrightarrow} F$, where $0$ is the trivial group.
\end{proposition}

\begin{proof}
    Recall $\displaystyle \lim_{\longrightarrow} F = \bigsqcup_{A \in S} \{A\} \times F(A) / \sim$, where $(A,\alpha) \sim (A', \alpha')$ if and only if $(F(\iota_A^B))(\alpha) = (F(\iota_{A'}^B))(\alpha')$ for some $B \in S$. Now consider the map $\psi: \displaystyle \lim_{\longrightarrow}{}^G F \to \displaystyle \lim_{\longrightarrow} F \times G$ sending $[A, \alpha, f]$ to $([A, \alpha], f)$, we want to show that $\psi$ is a well-defined isomorphism.

    
    First, $\psi$ is a well-defined bijection, following from that in $\displaystyle \lim_{\longrightarrow}{}^G F$, $(A, \alpha, f) \sim (A', \alpha', f')$ if and only if $(F(\iota_A^B))(\alpha) = (F(\iota_{A'}^B))(\alpha')$ for some $B \in \textup{Ob}(\mathcal{S}) = S$ and $f = f'$.
    
    Then we show that $\psi$ is a homomorphism. For any $[A, \alpha, f], [B, \beta, g] \in \displaystyle \lim_{\longrightarrow}{}^G F$, choose $C \in S$ such that $A \leq C$ and $f(B) \leq C$. Then
    \begin{align*}
        \psi([A, \alpha, f] [B, \beta, g]) =& \psi[(C, F(\iota_A^C)(\alpha), f) \circ (f^{-1}(C), F(\iota_B^{f^{-1}(C)})(\beta), g)]) \\
        =& \psi[C, F(\iota_A^C)(\alpha)F(\iota_B^{f^{-1}(C)})(\beta), fg] \\
        =& ([C, F(\iota_A^C)(\alpha)F(\iota_B^{f^{-1}(C)})(\beta))], fg) \\
        =& ([C, F(\iota_A^C)(\alpha)F(\iota_B^{C})(\beta))], fg) \quad (\textup{$f$ acts trivially on $\mathcal{S}$}) \\
        =& ([A, \alpha][B, \beta], fg) \\
        =& ([A, \alpha], f)([B, \beta], g)\\ =& \psi[A, \alpha, f] \psi[B, \beta, g]
    \end{align*}
\end{proof}

If $H$ is a subgroup of $G$, then the $G$-action on $\mathcal{S}$ naturally restricts to an $H$-action. In this case, $\mathcal{S}_H$ is a subcategory of $\mathcal{S}_G$, and the concentration structure on $\mathcal{S}_G$ restricts to the concentration structure on $\mathcal{S}_H$. This observation leads to the following lemma.

\begin{lemma}
    If $H$ is a subgroup of $G$, then $(\mathcal{S}_H, \sim)$ is a sub-concentration of $(\mathcal{S}_G, \sim)$.
\end{lemma}

Together with Proposition \ref{prop: stabilizatin of sub concentration is submonoid}, we have that

\begin{proposition} \label{prop: subgroup equivariant direct limit is subgroup}
     If $H$ is a subgroup of $G$, then $\displaystyle \lim_{\longrightarrow}{}^H F$ is a subgroup of $\displaystyle \lim_{\longrightarrow}{}^G F$.
\end{proposition}

\subsection{$G$-equivariant direct limit as a semidirect product}

In this subsection, we study how $G$-equivariant direct limits relate to semidirect products. First, we will show that the category $(\mathcal{S}_G, \sim)$ admits a natural decomposition as a semidirect product. It will allow us to express $G$-equivariant direct limit as a semidirect product of the ordinary direct limit with $G$.

\vspace{5mm}

Let $0$ denote the trivial group, and consider the category $\mathcal{S}_0$. In this case, $\textup{Mor}(A, B) = \{(A, \alpha, f): \alpha \in A, f\in 0, f(A)=B\}$, which is nonempty only when $A=B$. In particular, $\textup{Mor}(A, A) = \{(A, \alpha, \epsilon): \alpha \in F(A)\}$. For convenience, we simply denote $(A, \alpha, \epsilon)$ as $(A, \alpha)$. With this convention, $\displaystyle \lim_{\longrightarrow}{}^0 F$ and $\displaystyle \lim_{\longrightarrow} F$ are naturally identified via the isomorphism in Proposition \ref{prop: trivial G-equivariant semidirect product is direct product}.

\begin{example}
    Consider the directed set $S = \{C, D\}$ with $C \leq D$, and the functor $F: \mathcal{S} \to \mathscr{G}rp$ sending $C$ to $\mathbb{Z}/2$, $D$ to $\mathbb{Z}/4$ and sending $\iota_C^D$ to the multiplication by $2$. Then $(\mathcal{S}_0, \sim)$ is given by Example \ref{ex: Z/2 and Z/4 category with concentration}(1), where the computation shows that $\displaystyle \lim_{\longrightarrow}{}^0 F \cong \mathbb{Z}/2$.
\end{example}

Let $\mathcal{G}$ be the category containing a single object $*$ and $\textup{Mor}(*,*) = G$, with the discrete concentration structure. Given a $G$-action $\rho: G \to \textup{Aut}(\mathcal{S})$, we define a functor $\Phi: \mathcal{G} \to \mathcal{A}ut(\mathcal{S}_0)$, by $(\Phi(f))(A) =f(A)$ and $(\Phi(f))(A, \alpha) =(f(A),\alpha)$. Here we wrote $f(A)$ instead of $(\rho(f))(A)$ for simplicity.

It is natural to ask whether $\Phi$ is compatible with the concentrations. We give an affirmative answer to it.

\begin{lemma}
    $\Phi$ is compatible with the concentration structures on $\mathcal{S}_0$ and $\mathcal{G}$.
\end{lemma}

\begin{proof}
    The concentration structure on $\mathcal{G}$ is discrete, so it suffices to show that if $(A, \alpha) \sim (A', \alpha')$, then $(f(A), \alpha) \sim (f(A'), \alpha')$ for any $f \in G$.

    Indeed, $(A, \alpha) \sim (A', \alpha')$ implies that $(F(\iota_A^B))(\alpha) = (F(\iota_{A'}^B))(\alpha')$ for some $B$. Then $(F(\iota_{f(A)}^{f(B)}))(\alpha) = F(\iota_A^B)(\alpha) = (F(\iota_{A'}^B))(\alpha') = (F(\iota_{f(A')}^{f(B)}))(\alpha')$, so $(f(A), \alpha) \sim (f(A'), \alpha')$.
\end{proof}


Now we can take the semidirect product of $(\mathcal{S}_0,\sim)$ and $(\mathcal{G},\sim)$. The next theorem shows that this semidirect product is isomorphic to $(\mathcal{S}_G, \sim)$.

\begin{theorem} \label{thm: S_G is the semidirect product of S_0 and G}
    $(\mathcal{S}_G,\sim) \cong (\mathcal{S}_0,\sim) \rtimes_{\Phi} (\mathcal{G},\sim)$ as categories with concentration.
\end{theorem}

\begin{proof}

    Note that the only morphisms in $\mathcal{S}_0$ are $\textup{Mor}(A,A) = \{(A, \alpha): \alpha \in F(A)\}$, so the morphisms in $\mathcal{S}_0 \rtimes_{\Phi} \mathcal{G}$ are given by $\textup{Mor}((B,*),(A,*))=\{((A,\alpha),f): f\in G,\alpha\in F(A), f(B)=A\} $.
    
    Consider the functor $\Psi:\mathcal{S}_0 \rtimes_{\Phi} \G \rightarrow \mathcal{S}_G$ sending object $(A,*)$ to $A$ and morphism  $((A,\alpha),f)$ to $(A,\alpha,f)$. We want to show that $\Psi$ is a concentration isomorphism.

    Rephrasing the definition of concentration isomorphism in Definition \ref{def: concentration preserving functor}, we only need to check the following.
    \begin{enumerate}
        \item $\Psi$ is a bijective functor.
        \item $((A,\alpha),f) \sim_{\rtimes_{\Phi}}((A',\alpha'),f')$ if and only if $(A,\alpha,f) \sim (A',\alpha',f')$.
    \end{enumerate}

    The bijectivity is obvious. Next we check that $\Psi$ preserves the identities and the compositions. For any $A$, we have $\Psi((A,e_{F(A)}),\epsilon)=(A,e_{F(A)},\epsilon)=id_A$, so the identities are preserved. 

    To verify that $\Psi$ preserves the compositions, we take $((A,\alpha),f) \in \textup{Mor}((B,*),(A,*))$ and $((B,\beta),g) \in \textup{Mor}((C,*),(B,*))$. 
    Then
    \begin{align*}
        \Psi(((A,\alpha),f)\circ((B,\beta),g))&=\Psi((A,\alpha)\circ (\Phi(f))(B,\beta),fg) \\
        &= \Psi((A,\alpha)\circ (f(B),\beta),fg)\\
        &=\Psi((A,\alpha \beta),fg) \\
        &=(A,\alpha\beta,fg)\\
        &=(A,\alpha,f) \circ (B,\beta,g)\\
        &= \Psi((A,\alpha),f) \circ \Psi((B,\beta),g).
    \end{align*} 

    The second property to check comes from the following argument.
    \begin{align*}
        &((A,\alpha),f) \sim_{\rtimes_{\Phi}}((A',\alpha'),f')\\ \iff &(A,\alpha)\sim(A',\alpha') \textup{ and } f\sim f' \\
        \iff &(F(\iota_A^B))(\alpha) = (F(\iota_{A'}^B))(\beta) \textup { for some }B \in \textup{Ob}(\mathcal{S}_0)= \textup{Ob}(\mathcal{S}_G)\textup{, and } f=f'\\
        \iff&
        (A,\alpha,f) \sim (A',\alpha',f')
    \end{align*}
\end{proof}

In Section \ref{sec: stabiliztion monoids and semidirect products of monoids}, we have seen that the functor $\Phi: \mathcal{G} \to \mathcal{E}nd(\mathcal{S}_0)$ induces a group homomorphism $\phi: G = M_{(\mathcal{G}, \sim_{dis})} \to \textup{End}(M_{(\mathcal{S}_0, \sim)}) = \textup{End}(\displaystyle \lim_{\longrightarrow}{}^0 F)$, given by $(\phi(g))[A, \alpha] = [g(A), \alpha]$. Thus, the above theorem yields the following corollary.

\begin{corollary} \label{cor: G-direct limit is the semidirect product of direct limit and G}
    $\displaystyle \lim_{\longrightarrow}{}^G F \cong \displaystyle \lim_{\longrightarrow}{} F \rtimes_\phi G $
\end{corollary}

\begin{proof}
    Combine Proposition \ref{prop: concentration preserving functor induces monoid homomorphism}, Proposition \ref{prop: concentration of semidirect product is semidirect product of concentration} and Theorem \ref{thm: S_G is the semidirect product of S_0 and G}, we have
    $$
        \displaystyle \lim_{\longrightarrow}{}^G F = M_{(\mathcal{S}_G,\sim)} 
        \cong M_{(\mathcal{S}_0,\sim) \rtimes_{\Phi} (\mathcal{G},\sim)} 
        \cong M_{(\mathcal{S}_0,\sim)} \rtimes_{\phi} M_{(\mathcal{G},\sim)} 
        = \displaystyle \lim_{\longrightarrow}{}^0 F \rtimes_\phi G 
        = \displaystyle \lim_{\longrightarrow} F \rtimes_\phi G
    $$
\end{proof}

\subsection{Example: $\mathbb{R}$-braid groups}
\label{sec: real braid groups}

Let $S^{\mathbb{R}}$ be the collection of finite subsets of $\mathbb{R}$, with partial order subset inclusion $\subseteq$, then $(S^{\mathbb{R}}, \subseteq)$ is a directed set. Denote $\mathcal{S}^{\mathbb{R}}$ to be the associated direct category. We define the functor $Br: \mathcal{S}^{\mathbb{R}} \rightarrow \G rp$ as follows:

\begin{itemize}
    \item For any finite set $A$, let $Br(A) = Br_{|A|}$, the braid group with $|A|$ strands.
    \item For any pair $A \subseteq B$, let $Br(\iota_A^B): Br_{|A|} \to Br_{|B|}$ be the group homomorphism induced by adding backside trivial strands on $B \setminus A$, where we identify the endpoints of $Br_{|A|}$ and $Br_{|B|}$ with $A$ and $B$ respectively. See Figure \ref{fig: adding trivial strands to a real braid} for an illustration.
\end{itemize}

Let $\textup{Aut}^+(\mathbb{R})$ be the group of order-preserving automorphisms on $\mathbb{R}$. Here $\mathbb{R}$ can be viewed either as a set, a topological space or a smooth manifold. Then the corresponding $\textup{Aut}^+(\mathbb{R})$ is $\textup{Bij}^+(\mathbb{R})$, $\textup{Homeo}^+(\mathbb{R})$ or $\textup{Diff}^+(\mathbb{R})$ respectively.

Consider the $\textup{Aut}^+(\mathbb{R})$-action $\rho: \textup{Aut}^+(\mathbb{R}) \to \textup{Aut}(\mathcal{S})$ such that $\rho_f(A) = f(A)$ and $\rho_f(\iota_A^B) = \iota_{f(A)}^{f(B)}$. Again, we consider the trivial action on $\G rp$.

\begin{figure}
    \begin{subfigure}{0.4\textwidth}
        \centering
        \includegraphics{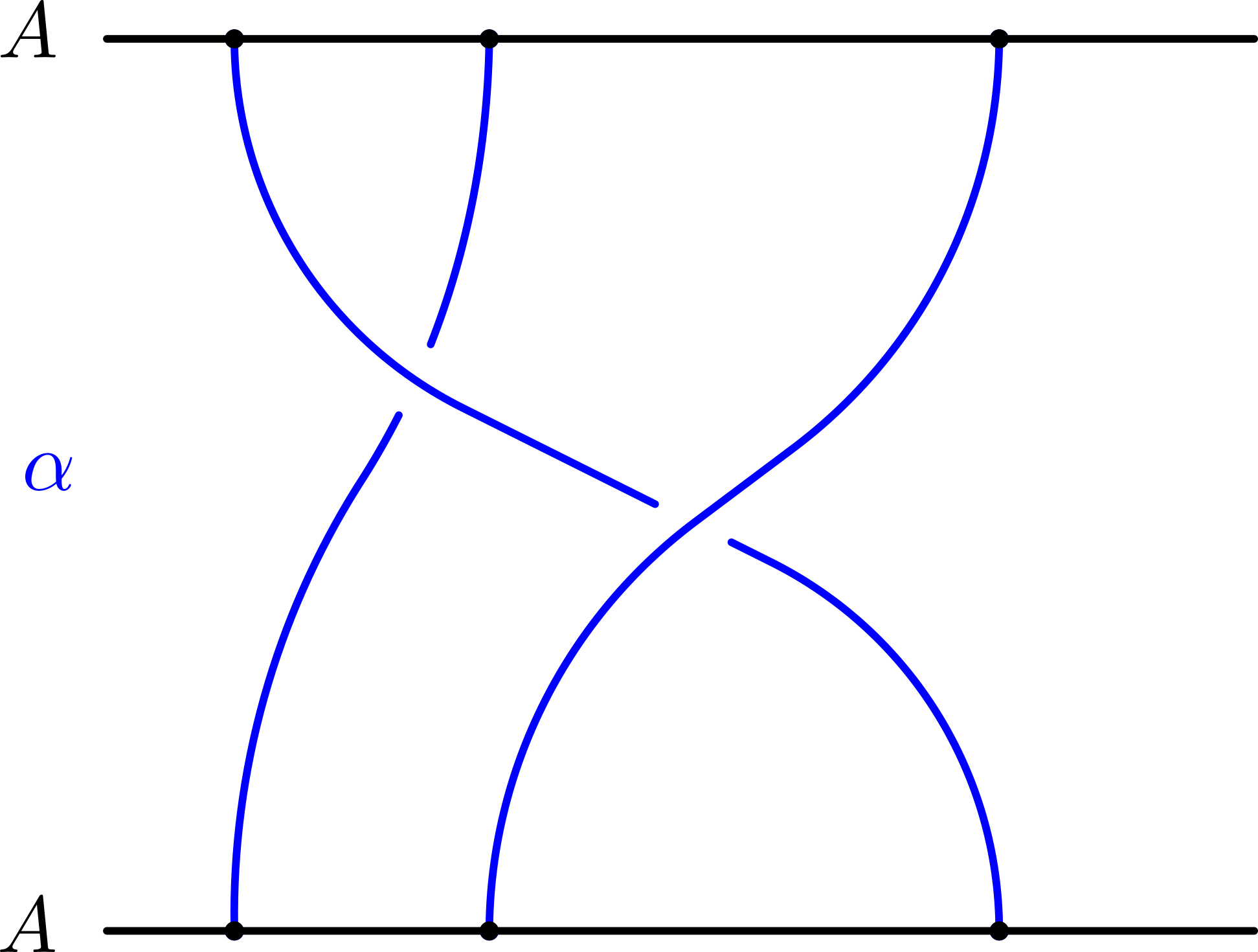}
        \caption{}
        \label{fig: real braid before adding trivial strands}
    \end{subfigure}
    \hfill
    \begin{subfigure}{0.4\textwidth}
        \centering
        \includegraphics{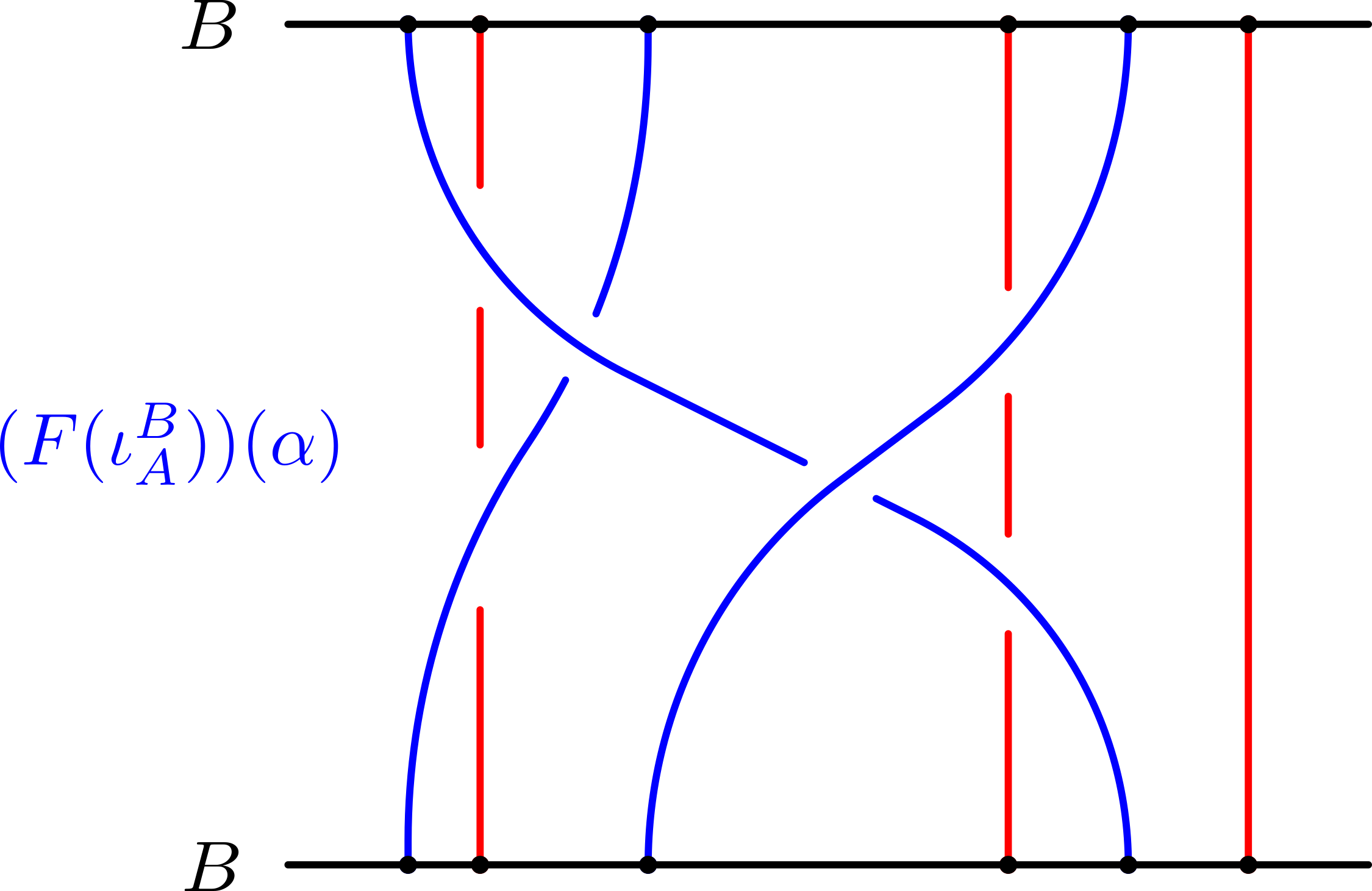}
        \caption{}
        \label{fig: real braid after after trivial strands}
    \end{subfigure}
    \caption{(a) An element $\alpha \in Br_{|A|}$. (b) The group homomorphism $F(\iota_A^B)$ adds backside trivial strands (marked in red) to $\alpha$. The endpoints of those trivial strands are on $B \setminus A$.}
    \label{fig: adding trivial strands to a real braid}
\end{figure}

\begin{lemma}
    $Br$ is an $\textup{Aut}^+(\mathbb{R})$-equivariant functor.
\end{lemma}

\begin{proof}
   It is easy to check that $Br$ is a functor. To show it is $\textup{Aut}^+(\mathbb{R})$-equivariant, we need to verify that $Br(A)=Br(f(A))$ and $Br(\iota_A^B)=Br(\iota_{f(A)}^{f(B)})$ for any $f \in \textup{Aut}^+(\mathbb{R})$ and any objects $A \subseteq B$.
   
   First, $Br(A) = Br_{|A|} = Br_{|f(A)|} = Br(f(A))$ follows from $f$ is a bijection, and $Br(\iota_A^B)=Br(\iota_{f(A)}^{f(B)})$ follows from that $f$ is also order preserving. To be more precise, $f$ is order preserving meaning that we are adding trivial strands in the same relative position, so the group homomorphism from $Br_{|A|}$ to $Br_{|B|}$ is the same as group homomorphism from $Br_{|f(A)|}$ to $Br_{|f(B)|}$.
\end{proof}

 Using the construction of $G$-equivariant direct limit (Definition \ref{def: G-equivariant direct limit}), we define the $\mathbb{R}$-braid group as follows.


\begin{definition}
    Define the \textbf{$\mathbb{R}$-braid group} $Br_\mathbb{R} := \displaystyle \lim_{\longrightarrow}{}^{\textup{Aut}^+(\mathbb{R})} Br$, the ${\textup{Aut}^+(\mathbb{R})}$-equivariant direct limit of $Br$.
\end{definition}

\begin{figure}
    \centering
    \begin{subfigure}[c]{0.4\textwidth}
        \centering
        \includegraphics{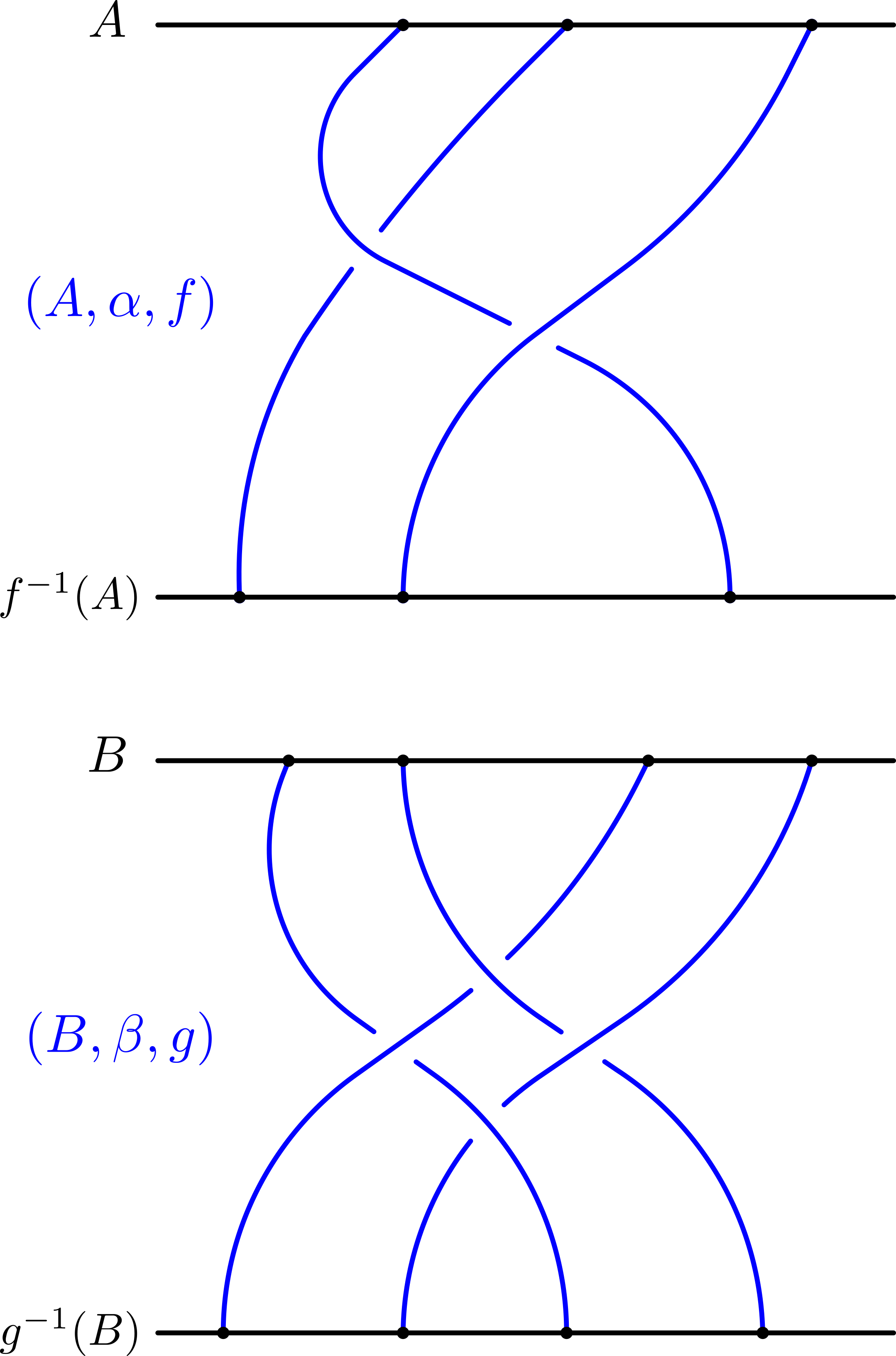}
        \caption{}
        \label{fig: two real braids with shifts}
    \end{subfigure}
    \hfill
    \begin{subfigure}[c]{0.45\textwidth}
        \centering
        \includegraphics{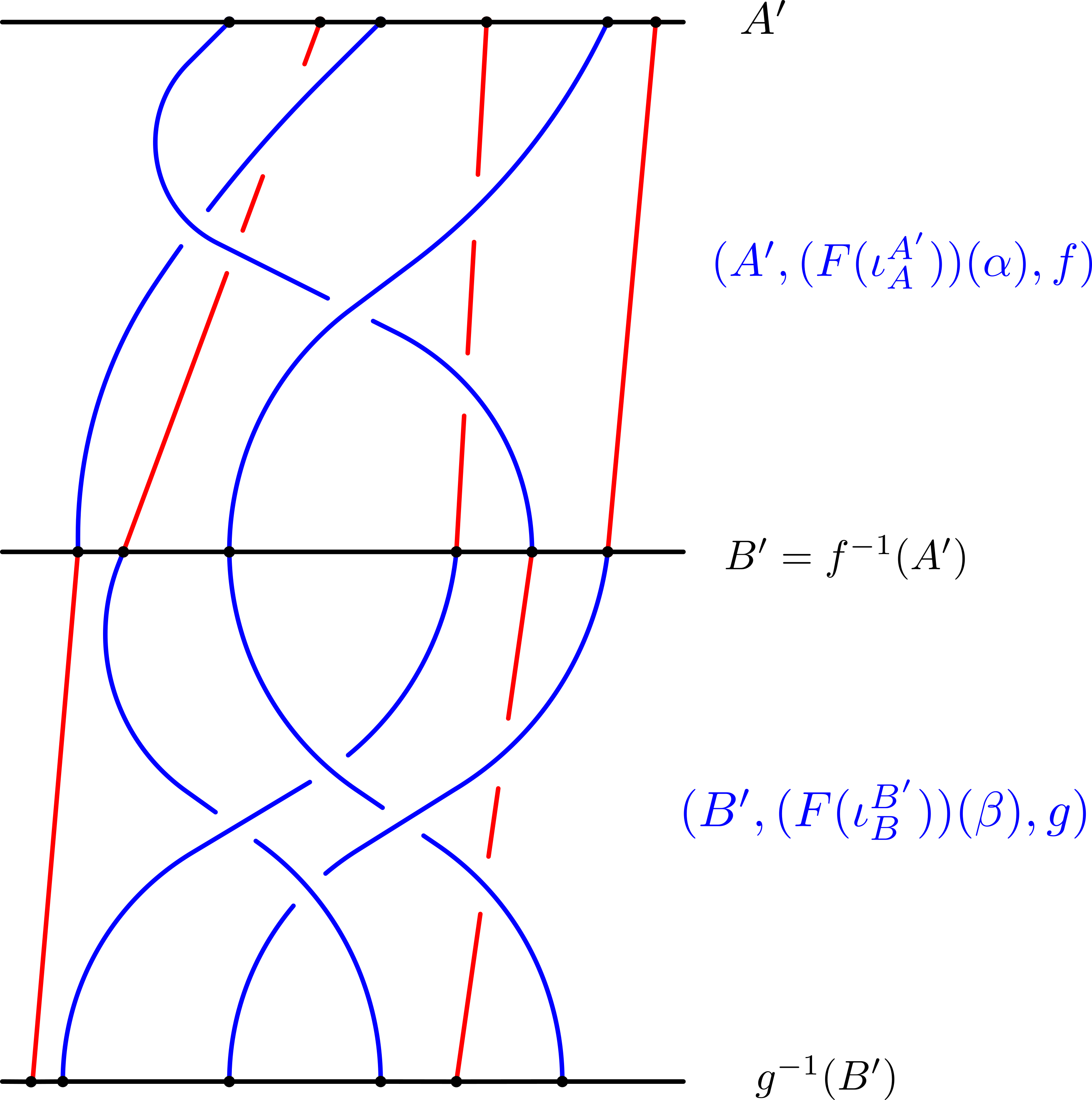}
        \vfill
        \caption{}
        \label{fig: product of real braids with shifts}
    \end{subfigure}
    \caption{(a) A diagram of $[A, \alpha, f]$ and a diagram of $[B, \beta, g]$. (b) To take the product of $[A, \alpha, f]$ and $[B, \beta, g]$, we add shifted trivial strands so that the two diagrams have matching endpoints in between. Their concatenation is the diagram of the product.}
    \label{fig: example of real braid group multiplication}
\end{figure}

By Corollary \ref{cor: G-direct limit is the semidirect product of direct limit and G}, $Br_\mathbb{R}$ is isomorphic to a semidirect product of $\displaystyle \lim_{\longrightarrow}{} Br$ and $\textup{Aut}^+(\mathbb{R})$.


Depending on the additional structures on $\mathbb{R}$, the group $\textup{Aut}^+(\mathbb{R})$ can be viewed as $\textup{Bij}^+(\mathbb{R})$, $\textup{Homeo}^+(\mathbb{R})$ or $\textup{Diff}^+(\mathbb{R})$. We denote the corresponding $\mathbb{R}$-braid groups as $Br_{\mathbb{R}}^{set}$ (the setwise $\mathbb{R}$-braid group), $Br_{\mathbb{R}}^{top}$ (the topological $\mathbb{R}$-braid group) and $Br_{\mathbb{R}}^{diff}$ (the smooth $\mathbb{R}$-braid group). By Proposition \ref{prop: subgroup equivariant direct limit is subgroup}, we have $Br_{\mathbb{R}}^{diff} \leq Br_{\mathbb{R}}^{top} \leq Br_{\mathbb{R}}^{set}$.

\vspace{5mm}


The definition of the $\mathbb{R}$-braid group $Br_{\mathbb{R}}$ as an $\textup{Aut}^+$-equivariant direct limit aligns naturally with the usual pictorial description of braid multiplication. An element $[A,\alpha,f] \in Br_{\mathbb{R}}$ can be represented by a braid $\alpha$ with top endpoints on $A$ and bottom endpoints on $f^{-1}(A)$. Such a representative is not unique: two diagrams representing the same class $[A,\alpha,f]$ may differ by adding and deleting some $f$-shifted trivial strands, each of which connects a point $p$ to $f(p)$, drawn behind all other strands.

The product of two elements $[A,\alpha,f]$ and $[B,\beta,g]$ is obtained by concatenating suitable representatives $(A', \alpha', f)$ and $(B', \beta', g)$ such that $f^{-1}(A') = B'$. See Figure \ref{fig: example of real braid group multiplication} for an example.

Equivalently, one may think of an element of $Br_\mathbb{R}$ as an infinite braid on the entire set $\mathbb{R}$, with only finitely many possibly ``non-trivial'' strands, and the product is given by direct concatenation along $\mathbb{R}$.


\section{Concentrations, fundamental groupoids and  fibrations}
\label{sec: concentrations fundamental groupoids and fibrations}


In this section, we present further applications of concentration structures, focusing on those arising from fundamental groupoids. By relating to fundamental groups and fibrations, we illustrate how concentration structures can capture and reinterpret classical topological invariants.

\subsection{Concentration structures on fundamental groupoids}\label{sec: Concentration structures on fundamental groupoids and fundamental groups}

We first recall the definition of fundamental groupoid.

\begin{definition}
    Let $X$ be a path-connected topological space. The fundamental groupoid $\Pi(X)$ is the groupoid whose objects are points in $X$, and the morphisms from $x$ to $y$ are the paths from $x$ to $y$ up to homotopy relative to endpoints. The composition of morphisms is induced by the concatenation of paths.
\end{definition}

By definition, the fundamental group of $X$ is recovered from $\textup{Mor}(x,x)$ for any $x\in X$. We will provide an alternative way of obtaining $\pi_1(X)$ from $\Pi(X)$ by equipping a concentration structure. 

First, choose a base point $x_0$ and a family of paths $\Theta_{x_0}=\{\theta_{x_0}^y| \ y\in X\}$, where each $\theta_{x_0}^y$ is a path from $x_0$ to $y$ up to homotopy relative to endpoints. In other words, $\theta_{x_0}^y$ is a morphism from $x_0$ to $y$ in the fundamental groupoid $\Pi(X)$. We require that $\theta_{x_0}^{x_0}$ is the constant path.

\begin{definition}
    Given a choice of $\Theta_{x_0}$ as above, we define a concentration structure $\sim$ on $\Pi(X)$ as follows: for any two morphisms $\alpha \in \textup{Mor}(a, b)$ and $\beta \in \textup{Mor}(c, d)$, define $\alpha \sim \beta$ if and only if $(\theta _{x_0}^b)^{-1}\circ\alpha\circ\theta_{x_0}^a=(\theta _{x_0}^d)^{-1}\circ\beta\circ\theta_{x_0}^c$, as morphisms from $x_0$ to itself.
\end{definition}

\begin{lemma}
    $\sim$ is indeed a concentration structure on $\Pi(X)$.
\end{lemma}

\begin{proof}
    It is easy to verify that $\sim$ is an equivalence relation. We now check that it satisfies the three axioms of $3$-concentration structure (Definition \ref{def: n-concentration}). Then it follows from Proposition \ref{prop: 3-concentration implies concentration} that $\sim$ is a concentration structure. 

    \textbf{Axiom 1.} For any point $a$ in $X$, let $id_{a}$ be the constant path at $a$. Then $(\theta_{x_0}^a)^{-1} \circ id_a \circ \theta_{x_0}^a = id_{x_0}$. Thus $id_a \sim id_b$ for any two points $a, b$ in $X$.
    
    \textbf{Axiom 2.} Given paths $\alpha:b\to c, \ \ \beta:a\to b$, $\alpha':b'\to c', \  \beta':a'\to b'$ such that $\alpha \sim \alpha', \ \beta \sim\beta'$. Then
    \begin{align*}
        (\theta_{x_0}^c)^{-1} \circ \alpha \circ \beta \circ \theta_{x_0}^a &= (\theta_{x_0}^c)^{-1} \circ \alpha  \circ \theta_{x_0}^b \circ (\theta_{x_0}^b)^{-1}  \circ \beta \circ \theta_{x_0}^a \\ 
        &= (\theta_{x_0}^{c'})^{-1} \circ \alpha'  \circ \theta_{x_0}^{b'} \circ (\theta_{x_0}^{b'})^{-1}  \circ \beta' \circ \theta_{x_0}^{a'}\\
        &=  (\theta_{x_0}^{c'})^{-1} \circ \alpha'  \circ \beta' \circ \theta_{x_0}^{a'}.
    \end{align*}
    showing that $\alpha \circ \beta \sim \alpha' \circ \beta'$.

    \textbf{Axiom 3.} Let $\alpha: a \to b, \beta: c \to d$ and $\gamma: e \to f$ be three paths in $X$. Note that $\alpha \sim (\theta_{x_0}^b)^{-1} \circ \alpha \circ \theta_{x_0}^a$, since $(\theta_{x_0}^b)^{-1} \circ \alpha \circ \theta_{x_0}^a = (\theta_{x_0}^{x_0})^{-1} \circ (\theta_{x_0}^b)^{-1} \circ \alpha \circ \theta_{x_0}^a \circ \theta_{x_0}^{x_0}$. Similarly, $\beta \sim (\theta_{x_0}^d)^{-1} \circ \alpha \circ \theta_{x_0}^c$ and $\gamma \sim (\theta_{x_0}^f)^{-1} \circ \alpha \circ \theta_{x_0}^e$. Moreover, $\left((\theta_{x_0}^b)^{-1} \circ \alpha \circ \theta_{x_0}^a\right) \circ \left((\theta_{x_0}^d)^{-1} \circ \alpha \circ \theta_{x_0}^c\right) \circ \left((\theta_{x_0}^f)^{-1} \circ \gamma \circ \theta_{x_0}^e\right)$ exists as a composition of three loops based at $x_0$.
\end{proof}

Theorem \ref{thm: concentration monoid of fundamental groupoid is fundamental group} follows directly from the following theorem.

\begin{theorem} \label{thm: concentration group is isomorphisc to fundamental group}
    For any choice of $\Theta_{x_0}$, the concentration group $M_{(\Pi(X), \sim)}$ is isomorphic to the fundamental group $\pi_1(X, x_0)$.
\end{theorem}

\begin{proof}
    Consider the map $\phi: M_{(\Pi(X), \sim)} \to \pi_1(X, x_0)$ sending $[\alpha]_{\sim}$ to $(\theta_{x_0}^b)^{-1} \circ \alpha \circ \theta_{x_0}^a$, where $\alpha$ is a path from $a$ to $b$. By the definition of $\sim$, the map $\phi$ is well-defined.

    We claim that $\phi$ is a group homomorphism. Indeed, given two paths $\alpha: a \to b$ and $\beta: c \to d$, we have
    \begin{align*}
        \phi([\alpha]_{\sim} [\beta]_{\sim}) &= \phi[(\theta_{x_0}^b)^{-1} \circ \alpha \circ \theta_{x_0}^a \circ (\theta_{x_0}^d)^{-1} \circ \beta \circ \theta_{x_0}^c] \\
        &= (\theta_{x_0}^{x_0})^{-1} \circ (\theta_{x_0}^b)^{-1} \circ \alpha \circ \theta_{x_0}^a \circ (\theta_{x_0}^d)^{-1} \circ \beta \circ \theta_{x_0}^c \circ \theta_{x_0}^{x_0} \\
        &= (\theta_{x_0}^b)^{-1} \circ \alpha \circ \theta_{x_0}^a \circ (\theta_{x_0}^d)^{-1} \circ \beta \circ \theta_{x_0}^c \\
        &= \phi[\alpha]_{\sim} \phi[\beta]_{\sim}
    \end{align*}

To see that $\phi$ is surjective, note that any $\alpha\in \pi_1(X,x_0)$ is the image of $[\alpha]_\sim\in M_{(\Pi(X),\sim)}$, since $\phi([\alpha]_\sim)= (\theta_{x_0}^{x_0})^{-1} \circ \alpha \circ \theta_{x_0}^{x_0}=\alpha$.

For injectivity, suppose $\alpha:a\to b$, $\alpha':a'\to b'$ satisfy 
$\phi([\alpha]_\sim)=\phi([\alpha']_\sim)$. Then the definition of $\phi$ implies that $(\theta_{x_0}^b)^{-1} \circ \alpha \circ \theta_{x_0}^a=(\theta_{x_0}^{b'})^{-1} \circ \alpha' \circ \theta_{x_0}^{a'}$, which exactly means $\alpha \sim \alpha'$. Hence $[\alpha]_\sim=[\alpha']_\sim$.
\end{proof}

Recall that the concentration structure we constructed depends on the choice of $\Theta_{x_0}$. We now show that this choice does not affect the concentration structure up to concentration isomorphism.

\begin{proposition}
    Up to concentration isomorphism, $(\Pi(X), \sim)$ does not depend on the choice of $\Theta_{x_0}$.
\end{proposition}

\begin{proof}
    Let $\Theta_{x_0}=\{\theta_{x_0}^y\}$ and $\Sigma_{x_0}=\{\sigma_{z_0}^y\}$ be two such choices of paths. Denote the corresponding concentration structures as $\sim_\theta$ and $\sim_\sigma$ respectively.
    
    Let $\rho$ be a path from $x_0$ to $z_0$. Consider the functor $\Phi: (\Pi(X), \sim_{\theta}) \to (\Pi(X), \sim_{\sigma})$, sending point $a$ to itself, sending path $\alpha: a \to b$ to $\sigma_{z_0}^b \circ \rho \circ (\theta_{x_0}^b)^{-1} \circ \alpha \circ \theta_{x_0}^a \circ \rho^{-1} \circ (\sigma_{z_0}^a)^{-1}$.

    It is obvious that $\Phi$ has a strong inverse $\Phi^{-1}: (\Pi(X), \sim_{\sigma}) \to (\Pi(X), \sim_{\theta})$, sending point $a$ to itself, sending path $\beta: a \to b$ to $\theta_{x_0}^b \circ \rho^{-1} \circ (\sigma_{z_0}^b)^{-1} \circ \beta \circ \sigma_{z_0}^a \circ \rho \circ (\theta_{x_0}^a)^{-1}$.

    Next we check that $\Phi$ is a concentration preserving functor. Suppose that $\alpha \sim_{\theta} \alpha'$, where $\alpha'$ is a path from $a'$ to $b'$. Then $(\theta_{x_0}^b)^{-1} \circ \alpha \circ \theta_{x_0}^a = (\theta_{x_0}^{b'})^{-1} \circ \alpha' \circ \theta_{x_0}^{a'}$. We have
    \begin{align*}
        { }&(\sigma_{z_0}^b)^{-1} \circ \Phi(\alpha) \circ \sigma_{z_0}^a \\
        =& (\sigma_{z_0}^b)^{-1} \circ \sigma_{z_0}^b \circ \rho \circ (\theta_{x_0}^b)^{-1} \circ \alpha \circ \theta_{x_0}^a \circ \rho^{-1} \circ (\sigma_{z_0}^a)^{-1} \circ \sigma_{z_0}^a \\
        =& \rho \circ (\theta_{x_0}^b)^{-1} \circ \alpha \circ \theta_{x_0}^a \circ \rho^{-1} \\
        =& \rho \circ (\theta_{x_0}^{b'})^{-1} \circ \alpha' \circ \theta_{x_0}^{a'} \circ \rho^{-1} \\
        =& (\sigma_{z_0}^{b'})^{-1} \circ \sigma_{z_0}^{b'} \circ \rho \circ (\theta_{x_0}^{b'})^{-1} \circ \alpha' \circ \theta_{x_0}^{a'} \circ \rho^{-1} \circ (\sigma_{z_0}^{a'})^{-1} \circ \sigma_{z_0}^{a'} \\
        =& (\sigma_{z_0}^{b'})^{-1} \circ \Phi(\alpha') \circ \sigma_{z_0}^{a'}
    \end{align*}

    Thus $\Phi(\alpha) \sim_{\sigma} \Phi(\alpha')$, and $\Phi$ is a concentration preserving functor. By a similar argument, we can show that $\Phi^{-1}$ is also a concentration preserving functor. Hence $\Phi$ is a concentration isomorphism.
\end{proof}

\subsection{Fibrations and pullback concentration structures}
\label{sec: fibrations and concentration structures on fundamental groupoids}

Given a continuous map $f: X \to Y$ between two path-connected topological spaces, $f$ induces a functor $\Pi(f): \Pi(X) \to \Pi(Y)$, sending a point $x\in X$ to $f(x)$, and a path $\alpha$ to $f(\alpha)$.


We now give a sufficient condition on $f : X \to Y$ that ensures the induced functor $\Pi(f)$ is $2$-lifting, thereby allowing us to pullback the concentration structures on the target space.

\begin{proposition}
 \label{prop: fibration induces $2$-lifting functor}
    Let $E,B$ be two path-connected topological spaces, and let $p: E\to B$ be a fibration. Then the functor $\Pi(p): \Pi(E) \to \Pi(B)$ is a $2$-lifting functor.
\end{proposition}


Before we prove this proposition, we first recall the definition of multivalued fibration. Lemma \ref{lem: surjective multivalued fibration is $2$-lifting} shows that any surjective multivalued fibration is a 2-lifting functor, then it suffices to prove that $\Pi(p)$ is a surjective multivalued fibration.

\begin{definition} \cite[Definition 3.2]{MR3538977} \label{def: multivalued fibration}
    A \textbf{multivalued fibration} is a functor $\mathcal{P}: \mathcal{E} \to \mathcal{B}$ such that for any morphism $g: B_0 \to B_1$ in $\mathcal{B}$ and any object $E_1$ in $\mathcal{E}$ with $\mathcal{P}(E_1) = B_1$, there exists (a not necessarily unique) $f: E_0 \to E_1$ satisfying $\mathcal{P}(f) = g$.
\end{definition}

\begin{lemma} \label{lem: surjective multivalued fibration is $2$-lifting}
    Any surjective multivalued fibration $\mathcal{P}: \mathcal{E} \to \mathcal{B}$ is $2$-lifting.
\end{lemma}

\begin{proof}
    Suppose that $g_2: A \to B, g_1: B \to C$ are two morphisms in $\mathcal{B}$. Note that $\mathcal{P}$ is surjective, so we can first find a morphism $f_1: \widetilde{B} \to \widetilde{C}$ in $\mathcal{E}$ such that $\mathcal{P}(f_1) = g_1$. Since $\mathcal{P}$ is a multivalued fibration, we can lift $g_2$ to some $f_2: \widetilde{A} \to \widetilde{B}$, with the target of $f_2$ and the source of $f_1$ shared. Then $f_1 \circ f_2$ exists.
\end{proof}

\begin{proof}[Proof of Proposition \ref{prop: fibration induces $2$-lifting functor}]
    We just need to prove that $\Pi(p)$ is a surjective multivalued fibration. 
    
    We first show the fibration $p$ is surjective. Fix a base point $e_0$ in $E$. For any point $b \in B$, because $B$ is path-connected, there exists a path $\gamma$ from $b$ to $p(e_0)$. Since $p$ is a fibration, we can lift $\gamma$ to a path $\widetilde{\gamma}$ from some $\tilde{b} \in E$ to $e_0$, then $p(\tilde{b}) = b$.

    For any path $\beta: b_0 \to b_1$ in $B$, we can find \textit{some} point $e_1$ in $E$ such that $p(e_1) = b_1$ by the surjectivity of $p$. Since $p$ is a fibration, there exists some path $\alpha: e_0 \to e_1$ such that $p(\alpha) = \beta$. Consider $\alpha$ and $\beta$ as morphisms in the fundamental groupoids, we have $\Pi(p)(\alpha) = \beta$, meaning $\Pi(p)$ is surjective.

    Lastly we check that $\Pi(p)$ is a multivalued fibration. For any path $\beta: b_0 \to b_1$ in $B$ and \textit{any} point $e_1'$ in $E$ such that $p(e_1') = b_1$, the path $\beta$ can be lifted to some path $\alpha: e_0 \to e_1'$. Then $\Pi(p)(\alpha) = \beta$, meaning $\Pi(p)$ is a multivalued fibration.
\end{proof}

\begin{remark}
    It is well known that if $p$ is a fibration, then $\Pi(p)$ is a Grothendieck fibration, which is stronger than being a multivalued fibration. In the proof of Proposition \ref{prop: fibration induces $2$-lifting functor} we give a self-contained proof that $\Pi(p)$ is a multivalued fibration, because multivalued fibration is sufficient to guarantee the $2$-lifting property.
    
    It is worth noting that any surjective Grothendieck fibration is $2$-lifting, and hence can pullback the concentration structures on the base category.
\end{remark}


With these results, we can now combine the concentration structures on fundamental groupoids with the pullback property for $2$-lifting functors to prove Theorem \ref{thm: any group can be recovered as concentration group of trivial category}.

\begin{proof}[Proof of Theorem \ref{thm: any group can be recovered as concentration group of trivial category}]
    For any group $G$, choose a path-connected topological space $X$ with $\pi_1(X)\cong G$. 

    By Theorem \ref{thm: concentration monoid of fundamental groupoid is fundamental group}, there exists a concentration structure $\sim$ on $\Pi(X)$ with concentration group $M_{(\Pi(X),\sim)} \cong G$. Let $\tilde{X}$ be the universal cover of $X$, with covering map $p:\tilde{X} \to X$. In particular, $p$ is a fibration. By Proposition \ref{prop: fibration induces $2$-lifting functor}, $p$ induces a $2$-lifting functor $\Pi(p):\Pi(\tilde{X}) \to \Pi(X)$. Then $\Pi(\tilde{X})$ admits a pullback concentration structure $\sim^{p^*}$. Lemma \ref{lem: pullback concentration induce isomorphism on concentration monoids} implies that
    $$
    M_{(\Pi(\tilde{X}),\sim^{p^*})} \cong M_{(\Pi(X),\sim)} \cong G
    $$
    
    Since $\tilde{X}$ is simply connected, it is easy to verify that $\Pi(\tilde{X})$ is equivalent to the trivial category.
\end{proof}

\bibliographystyle{alpha}
\bibliography{bib}

\begin{thebibliography}{BGHL14}

\bibitem[Art47]{MR19087}
E.~Artin.
\newblock Theory of braids.
\newblock {\em Ann. of Math. (2)}, 48:101--126, 1947.

\bibitem[BBP99]{MR1725510GeneralizedCongruences}
Marek~A. Bednarczyk, Andrzej~M. Borzyszkowski, and Wieslaw Pawlowski.
\newblock Generalized congruences---epimorphisms in {$Cat$}.
\newblock {\em Theory Appl. Categ.}, 5:No. 11, 266--280, 1999.

\bibitem[BCL07]{MR2494910}
Paolo Bertozzini, Roberto Conti, and Wicharn Lewkeeratiyutkul.
\newblock Non-commutative geometry, categories and quantum physics.
\newblock {\em East-West J. Math.}, pages 213--259, 2007.

\bibitem[BCL11]{MR2735769}
Paolo Bertozzini, Roberto Conti, and Wicharn Lewkeeratiyutkul.
\newblock A horizontal categorification of {G}el'fand duality.
\newblock {\em Adv. Math.}, 226(1):584--607, 2011.

\bibitem[BGHL14]{MR3319701}
Anna Beliakova, Zaur Guliyev, Kazuo Habiro, and Aaron~D. Lauda.
\newblock Trace as an alternative decategorification functor.
\newblock {\em Acta Math. Vietnam.}, 39(4):425--480, 2014.

\bibitem[Boh47]{MR19088}
F.~Bohnenblust.
\newblock The algebraical braid group.
\newblock {\em Ann. of Math. (2)}, 48:127--136, 1947.

\bibitem[CF94]{MR1295461}
Louis Crane and Igor~B. Frenkel.
\newblock Four-dimensional topological quantum field theory, {H}opf categories, and the canonical bases.
\newblock volume~35, pages 5136--5154. 1994.
\newblock Topology and physics.

\bibitem[Cra95]{Crane:1995qj}
Louis Crane.
\newblock {Clock and category: Is quantum gravity algebraic?}
\newblock {\em J. Math. Phys.}, 36:6180--6193, 1995.

\bibitem[EL16]{MR3448186}
Ben Elias and Aaron~D. Lauda.
\newblock Trace decategorification of the {H}ecke category.
\newblock {\em J. Algebra}, 449:615--634, 2016.

\bibitem[FJR11]{MR2805998}
Stefan Friedl, Andr\'as Juh\'asz, and Jacob Rasmussen.
\newblock The decategorification of sutured {F}loer homology.
\newblock {\em J. Topol.}, 4(2):431--478, 2011.

\bibitem[Kel82]{MR651714}
Gregory~Maxwell Kelly.
\newblock {\em Basic concepts of enriched category theory}, volume~64 of {\em London Mathematical Society Lecture Note Series}.
\newblock Cambridge University Press, Cambridge-New York, 1982.

\bibitem[Kho00]{MR1740682}
Mikhail Khovanov.
\newblock A categorification of the {J}ones polynomial.
\newblock {\em Duke Math. J.}, 101(3):359--426, 2000.

\bibitem[Lei98]{leinster1998basicbicategories}
Tom Leinster.
\newblock Basic bicategories.
\newblock {\em preprint}, 1998.

\bibitem[Lur09]{MR2522659}
Jacob Lurie.
\newblock {\em Higher topos theory}, volume 170 of {\em Annals of Mathematics Studies}.
\newblock Princeton University Press, Princeton, NJ, 2009.

\bibitem[Man19]{MR3900777}
Andrew Manion.
\newblock On the decategorification of {O}zsv\'ath and {S}zab\'o's bordered theory for knot {F}loer homology.
\newblock {\em Quantum Topol.}, 10(1):77--206, 2019.

\bibitem[Rob14]{MR3312621}
Lawrence~P. Roberts.
\newblock The decategorification of bordered {K}hovanov homology.
\newblock {\em J. Knot Theory Ramifications}, 23(14):1450078, 33, 2014.

\bibitem[Ste99]{MR1715405}
Benjamin Steinberg.
\newblock Semidirect products of categories and applications.
\newblock {\em J. Pure Appl. Algebra}, 142(2):153--182, 1999.

\bibitem[Wei17]{MR3538977}
Ittay Weiss.
\newblock Metric constructions of topological invariants.
\newblock {\em Topology Proc.}, 49:85--104, 2017.

\end{thebibliography}







\end{document}